\documentclass[A4paper,twocolumn]{article}

\usepackage{amsfonts,amsmath,amsthm}
\usepackage{amssymb}
\usepackage[colorlinks=true]{hyperref}
\usepackage{graphicx}

\newcommand{\alphau}{u}
\newcommand{\Tr}[1]{\mathrm{Tr}(#1)}
\newcommand{\phiR}{\phi_{{r}}}

\newcommand{\nt}{{\bar n_{\mathrm{ph}}}}

\newcommand{\superoperEvolution}{\mathbb{T_{\text{}}}}
\newcommand{\superoperInjection}{\mathbb{D_{\text{}}}}

\newcommand{\superoperMeasurement}{\mathbb{P_{\text{}}}}
\newcommand{\bid}{\text{\bf{I}}}
\newcommand{\ba}{\text{\bf{a}}}
\newcommand{\bN}{\text{\bf{N}}}

\newcommand{\nth}{n_{\text{th}}}
\newcommand{\nph}{n_{\text{ph}}}

\newtheorem{theorem}{Theorem}[section]%

\newtheorem{lem}{Lemma}[section]
\newtheorem{assum}{Assumption}

\newcommand{\RR}{{\mathbb R}}
\newcommand{\Id}{\bid}

\newcommand{\LL}{{\mathbb L}}
\newcommand{\bL}{{\mathbf L}}
\newcommand{\EE}[1]{{\text{\normalsize$\mathbb E$}}\left(#1\right)}

\newcommand{\CC}{{\mathbb C}}
\newcommand{\DD}{{\mathcal D}}

\newcommand{\HH}{{\mathcal H}}
\newcommand{\MM}{{\mathbb M}}

\newcommand{\KK}{{\mathbb K}}
\newcommand{\II}{{\Id}}

\newcommand{\bra}[1]{\left<#1\right|}
\newcommand{\ket}[1]{\left|#1\right>}
\newcommand{\tr}[1]{\mathrm{Tr}\left(#1\right)}

\newcommand{\bket}[1]{\left<#1\right>}

\newcommand{\nmax}{n_{\text{ph}}^{\text{\tiny max}}}

\newcommand{\hrho}{\widehat{\rho}}
\newcommand{\rhoe}{\rho^{\text{\tiny est}}}
\newcommand{\hrhoe}{\widehat{\rho}^{\text{\tiny est}}}
\newcommand{\hchi}{\widehat{\chi}}

\newcommand{\mur}{\mu^{\prime}}
\newcommand{\mr}{m^{\prime}}

\begin{document}

\title{Feedback stabilization of discrete-time  quantum systems subject to  non-demolition measurements with  imperfections and delays
\thanks{This work was supported in part by the "Agence Nationale de la Recherche" (ANR), Project QUSCO-INCA, Projet Jeunes Chercheurs EPOQ2 number ANR-09-JCJC-0070 and Projet Blanc CQUID number 06-3-13957, and by the EU under the IP project AQUTE and ERC project DECLIC. This work was also performed during the post-doctoral fellowship of A. Somaraju at Mines-ParisTech and INRIA during academic year 2010/2011.}}

\author{ H. Amini\thanks{Centre Automatique et Syst\`{e}mes, Math\'{e}matiques et Syst\`{e}mes, Mines ParisTech,
60 Boulevard Saint-Michel, 75272 Paris Cedex 6, France.}
\and
R. Somaraju\thanks{B-Phot, Department of Applied Physics and Photonics, Vrije Universiteit Brussel, Pleinlaan no. 2, Brussels - 1050 (Belgium).}
\and 
I. Dotsenko\thanks{Laboratoire Kastler-Brossel, ENS, UPMC-Paris 6, CNRS, 24 rue Lhomond, 75005 Paris, France.}
\and 
C. Sayrin\thanks{Laboratoire Kastler-Brossel, ENS, UPMC-Paris 6, CNRS, 24 rue Lhomond, 75005 Paris, France.}
\and 
M. Mirrahimi\thanks{INRIA Paris-Rocquencourt, Domaine de Voluceau, BP 105, 78153 Le Chesnay Cedex, France.}
\and 
P. Rouchon\thanks {Centre Automatique et Syst\`{e}mes, Math\'{e}matiques et Syst\`{e}mes, Mines ParisTech, 60 Boulevard Saint-Michel, 75272 Paris Cedex 6, France.}
 }
\maketitle


%%%%%%%%%%%%%%%%%%%%%%%%%%%%%%%%%%%%%%%%%%%%%%%%%%%%%%%%%%%%%%%%%%%%%%%%%%%%%%%%

\begin{abstract}

We consider  a controlled   quantum system  whose finite dimensional state is governed by a discrete-time nonlinear Markov process. In open-loop, the measurements are assumed to be quantum non-demolition (QND). The eigenstates of the measured observable are thus the open-loop stationary states: they are used to construct  a closed-loop supermartingale playing the role of a strict control Lyapunov function. The parameters of this supermartingale are  calculated by inverting a Metzler matrix that characterizes the impact of the control input on the Kraus operators defining the Markov process. The resulting state feedback scheme, taking into account  a known constant delay, provides the almost sure convergence of the controlled system to the target state. This convergence is ensured even in the case where the filter equation results from imperfect measurements corrupted by random errors  with  conditional probabilities given as  a  left stochastic matrix. Closed-loop simulations corroborated by experimental data illustrate the interest of such nonlinear feedback scheme for the photon box, a cavity quantum electrodynamics  system.

\end{abstract}

\section{Introduction}

Manipulating quantum systems allows one to accomplish tasks far beyond the reach of classical devices. Quantum information is paradigmatic in this sense: quantum computers will substantially outperform classical machines for several problems~\cite{nielsen-chang-book}. Though significant progress has been made recently, severe difficulties  still remain, amongst which decoherence is certainly the most important. Large systems consisting of many qubits must be prepared in fragile quantum states, which are rapidly destroyed by their unavoidable coupling to the environment. Measurement-based feedback and   coherent feedback  are  possible routes towards the preparation, protection  and stabilization of such states. For coherent feedback strategy,  the controller is also a quantum system coupled to the original one (see~\cite{gough-james:ieee09,gough-james:ieee10} and the references therein).  This paper is devoted to measurement-based feedback
where the controller and the control input are classical objects~\cite{wiseman-milburn:book}. The results presented  here  are  directly inspired  by   a recent experiment~\cite{sayrin-et-al:nature2011,sayrin:thesis}  demonstrating that such a quantum feedback scheme achieves   the on-demand preparation and stabilization of non-classical states of a  microwave field.

Following~\cite{geremia:PRL06} and relying on continuous-time Lyapunov techniques exploited in~\cite{mirrahimi-handel:siam07}, an initial measurement-based feedback  was proposed in~\cite{dotsenko-et-al:PRA09}. This feedback scheme  stabilizes  photon-number states (Fock states) of a microwave field (see e.g. \cite{haroche-raimond:book06} for a physical description of such cavity quantum electrodynamics (CQED) systems). The controller consists of a quantum filter that estimates   the  state of the field   from  discrete-time measurements performed by   probe atoms, and  secondly a stabilizing state-feedback that relies  on Lyapunov techniques. The discrete-time behavior is   crucial   for  a possible   real-time  implementation  of such  controllers. Closed-loop simulations reported  in~\cite{dotsenko-et-al:PRA09} have been confirmed by  the stability analysis performed  in~\cite{amini-et-al:ieee12}.  In the experimental implementation~\cite{sayrin-et-al:nature2011,sayrin:thesis}, the state-feedback has been improved by considering a strict Lyapunov function: this ensures better convergence properties by avoiding the passage by high photon numbers during the transient regime.  The goal of this paper is to present, for  a class  of discrete-time quantum systems,   the mathematical methods  underlying such improved Lyapunov design. Our  main result is given in Theorem~\ref{thm:obscontrollerMeas} where closed-loop  convergence  is proved  in presence of delays and measurement imperfections.

The state-feedback scheme  may be applied to generic discrete-time finite-dimensional  quantum systems using \emph{controlled} measurements in order to deterministically prepare and stabilize the system at some pre-specified target state. The dynamics of these systems may be expressed in terms of a (classical) nonlinear controlled Markov chain, whose state space consists of the set of density matrices on some Hilbert space. The jumps in this Markov chain are induced by quantum measurements and the associated jump probabilities are state-dependent. These systems  are  subject to a discrete-time sequence of positive operator valued measurements (POVMs \cite{nielsen-chang-book,haroche-raimond:book06}) and we use these POVMs to stabilize the system at the target state. By controlled measurements, we mean that at each time-step the chosen POVM is not fixed but is a function of some   classical control signal $u$, similar to~\cite{wiseman-doherty-PRL05}.  However, we  assume that  when the control $u$ is zero, the chosen POVM performs a quantum non-demolition (QND) measurement~\cite{haroche-raimond:book06,wiseman-milburn:book} for some orthonormal basis that includes the target state. The feedback-law  is based on a Lyapunov function that is a linear combination of a set of martingales corresponding to the open-loop QND measurements. This Lyapunov function  determines a ``distance'' between the target state and the current state. The parameters of this Lyapunov function are given by inverting Metzler  matrices characterizing the impact of the control input on the Kraus operators defining the Markov processes and POVMs. The (graph theoretic) properties of the Metzler  matrices are used to construct families of  open-loop supermartingales  that become strict supermartingales in closed-loop. This fact provides directly the convergence to the target state without using the invariance principle.

A common problem that occurs in quantum feedback control is that of delays between the measurement process and the control process~\cite{Nishio2009,kashima-yamamoto:ieee09}. In this paper we demonstrate, using a predictive quantum filter, that the proposed scheme works even in the presence of delays. Convergence analysis is done for perfect and imperfect  measurements.  For imperfect measurements, the dynamics of the system are governed by a nonlinear Markov chain given in~\cite{somaraju-et-al:acc2012}.

In both the perfect and imperfect measurement situations we prove a robustness property of the feedback algorithm: the convergence of the closed-loop system is ensured even when the feedback law is based on the state of a quantum filter that is not initialized correctly. This robustness property, is similar in spirit to the separation principle proven in~\cite{bouten-handel:2008,bouten-handel:2009}. We use the fact that the state space is a convex set and the target state, being a pure state, is an extreme point of this convex state space and therefore  cannot be expressed as a convex combination of any other states. One then uses the linearity of the conditional expectation to prove the robustness property. Our result is only valid for  target quantum states that are pure states.

The paper is organized as follows. In Section~\ref{sec:openloop}, we describe the finite dimensional  Markov model together with the main modeling assumptions in the case of perfect measurements and study the open-loop behavior (Theorem~\ref{thm:first}) which can be seen as a  non-deterministic protocol for  preparing a finite number of isolated and orthogonal  quantum states. In Section~\ref{sec:closedloop}, we  present  the main ideas underlying the construction of these control-Lyapunov functions $W_\epsilon$: a Metzler matrix  attached to the second derivative of the measurement operators and a technical lemma  assuming this Metzler matrix is irreducible. Finally,  Theorem~\ref{thm:main} describes the stabilizing state feedback derived from $W_\epsilon$. The same analysis is done for the case of imperfect measurements in Section~\ref{sec:imperfect}. For the state estimations used in the feedback scheme we propose a  brief discussion on the quantum filters and  prove a rather general robustness property for perfect measurements in Section~\ref{sec:closedloop} with Theorem~\ref{thm:obscontroller} and for imperfect ones in Section~\ref{sec:imperfect} with Theorem~\ref{thm:obscontrollerMeas}. Section~\ref{sec:photonbox} is devoted to the experimental implementation that has been done at Laboratoire Kastler-Brossel of  Ecole Normale Sup\'{e}rieure de Paris.  Closed-loop simulations and experimental data  complementary to those reported in~\cite{sayrin-et-al:nature2011,sayrin:thesis} are presented.
\section{System model and open-loop dynamics} \label{sec:openloop}
\subsection{The nonlinear Markov model} \label{ssec:model}
We consider a finite dimensional quantum system (the underlying Hilbert space $\HH=\CC^d$ is of  dimension $d>0$)
being measured through a generalized measurement procedure at discrete-time intervals.
The dynamics  are described by a nonlinear controlled  Markov chain. Here, we suppose perfect measurements and no decoherence.
The system state is described by a density operator $\rho$ belonging to $\DD$,  the set of non-negative Hermitian matrices of trace one:
$
\DD :=\{\rho\in\CC^{ d\times d}~|\quad\rho=\rho^\dag,\quad\tr{\rho}=1,\quad\rho\geq 0\}.
$
To each measurement outcome $\mu\in \{1,\ldots,m\}$, $m$ being the numbers of possible outcomes,  is attached
the Kraus operator $M_{\mu}^u\in \CC^{ d\times d}$ depending on $\mu$ and  on a scalar control input $u\in\RR$. For each $u$,
 $(M_\mu^u)_{\mu\in\{1,\ldots,m\}}$ satisfy the constraint $\sum_{\mu=1}^{m}{M_{\mu}^u}^\dag M_{\mu}^u=\II$, the
identity matrix. The Kraus map $\KK^u$ is defined by
 \begin{equation}\label{eq:kraus}
 \DD \ni \rho \mapsto \KK^u(\rho)= \sum_{\mu=1}^m M_\mu^u \rho {M_\mu^u}^\dag \in \DD
 .
 \end{equation}
The random evolution of the state $\rho_k\in\DD $ at time-step $k$ is  modeled through the following dynamics:
\begin{equation}\label{eq:dynamic}
 \rho_{k+1}=\MM_{\mu_k}^{u_{k-\tau}}(\rho_k):=
     \frac{M_{\mu_k}^{u_{k-\tau}}\rho_k  \left(M_{\mu_k}^{u_{k-\tau}}\right)^\dag}
    {\tr{M_{\mu_k}^{u_{k-\tau}}\rho_k \left(M_{\mu_k}^{u_{k-\tau}}\right)^\dag}},
\end{equation}
where $u_{k-\tau}$ is the control at step $k$, subject to a delay of $\tau>0$ steps. This delay is usually due to  delays in the measurement process that  can also be seen as delays in the control process.  $\mu_k$ is a random variable taking values $\mu$ in $\{1,\ldots,m\}$  with probability $p_{\mu,\rho_k}^{u_{k-\tau}}=\tr{M_{\mu}^{u_{k-\tau}}\rho_k \left(M_{\mu}^{u_{k-\tau}}\right)^\dag}$. For each $\mu$,  $\MM_{\mu}^u(\rho)$ is  defined  when  $p_{\mu,\rho}^u=\tr{M_{\mu}^u\rho \left(M_{\mu}^u\right)^\dag} \neq 0$.

We now state some assumptions that we will be using in the remainder of this paper.
\begin{assum}\label{assum:imp}
For $u=0$, all $M_\mu^0$ are diagonal in the same orthonormal  basis $\{\, \ket{n}\, |\, n\in\{1,\ldots, d\}\}$: $M_\mu^0=\sum_{n=1}^{ d}c_{\mu,n}\ket{n}\bra{n}$
with $c_{\mu,n}\in\CC.$
\end{assum}
\begin{assum}\label{assum:impt}
 For all $n_1\neq n_2$ in $\{1,\ldots, d\},$ there~exists $\mu\in\{1,\ldots,m\}$ such that $|c_{\mu,n_1}|^2\neq |c_{\mu,n_2}|^2.$
\end{assum}
\begin{assum}\label{assum:imo} All $M_\mu^u$ are  $C^2$ functions of $u$.
\end{assum}
Assumption~\ref{assum:imp} means that when $u=0$ the  measurements are
quantum non demolition (QND) measurements  over the states
$\bigl\{\ket{n}\,|\quad n\in\{1,\ldots, d\}\bigr\}$: when $u_k\equiv 0$, any $\rho=\ket{n}\bra{n}$
(orthogonal projector on the basis vector $\ket{n}$)  is a fixed point of~\eqref{eq:dynamic}.
Since $\sum_{\mu=1}^{m}{M_{\mu}^0}^{\dag}M_{\mu}^0=\II$, we have
$\sum_{\mu=1}^{m}|c_{\mu,n}|^2=1$ for all $n\in\{1,\ldots, d\}$ according to Assumption~\ref{assum:imp}.
Assumption~\ref{assum:impt} means that there exists a $\mu$ such that the statistics  when  $u_{k}\equiv 0$  for obtaining
the measurement result $\mu$ are different for the fixed points $\ket{n_1}\bra{n_1}$  and $\ket{n_2}\bra{n_2}$.  This follows by
noting that $\tr{M_{\mu}^0\ket{n}\bra{n}{M_{\mu}^0}^\dag}=|c_{\mu,n}|^2$ for $n\in\{1,\ldots, d\}$.
Assumption~\ref{assum:imo} is a technical assumption we will use in our proofs.
\subsection{Convergence of the open-loop dynamics}\label{ssec:openloop}
When the control input vanishes ($u\equiv 0$), the dynamics are simply given by
\begin{equation}\label{eq:dynopen}
\rho_{k+1} = \MM_{\mu_k}^0(\rho_k),
\end{equation}
where $\mu_k$ is a random variable with values in $\{1,\ldots,m\}$. The probability $p_{\mu,\rho_k}^{0}$ to have $\mu_k=\mu$ depends
on $\rho_k$: $p_{\mu,\rho_k}^0=\tr{M_{\mu}^{0}\rho_k {M_{\mu}^{0}}^\dag}$.
\begin{theorem}\label{thm:first}
Consider a Markov process $\rho_k$ obeying the dynamics of~\eqref{eq:dynopen} with an initial condition $\rho_0$ in $\DD $. Then
 with probability one, $\rho_k$ converges  to one of the $ d$ states $\ket{n}\bra{n}$ with $n\in\{1,\ldots, d\}$ and
the probability of convergence towards the state $\ket{n}\bra{n}$  is given by
$\bra{n}\rho_0\ket{n}.$
\end{theorem}
This Theorem is already proved in~\cite{amini-et-al:ieee12,amini-et-al:cdc2011} and also in~\cite{Bauer2011} in a slightly different formulation. A direct proof is based on the following Lyapunov function:
\begin{equation}\label{eq:V}
\Gamma(\rho):=-\sum_{n=1}^{ d}\frac{\big(\bra{n}\rho\ket{n}\big)^2}{2}.
\end{equation}
Since  $\bra{n}\rho\ket{n}$ is a martingale for any $n$ and since
\begin{multline}\label{eq:Q1}
  \Gamma(\rho_k) -  \EE{\Gamma(\rho_{k+1})|\rho_k}=Q_1(\rho_k) :=
  \\ \sum_{n,\mu,\nu}\tfrac{p_{\mu,\rho_k}^0p_{\nu,\rho_k}^0}{4}
\left(\tfrac{|c_{\mu,n}|^2\bra{n}\rho_k\ket{n}}{p_{\mu,\rho_k}^0}-\tfrac{|c_{\nu,n}|^2\bra{n}\rho_k\ket{n}}{p_{\nu,\rho_k}^0}\right)^2
\end{multline}

\section{Feedback stabilization with perfect measurements}\label{sec:closedloop}
In  Subsection~\ref{ssec:overview},  we give an overview of the control method and then in Subsections~\ref{ssec:weights} and~\ref{ssec:stab}, we prove the main results. Finally in Subsection~\ref{ssec:Qfilter}, we prove a robustness principle that explains how we can ensure convergence even if the initial state $\rho_0$ is unknown.
\subsection{Overview of the control method}\label{ssec:overview}
Theorem~\ref{thm:first} shows that the open-loop dynamics are stable in the sense that in each realization  $\rho_k$ converges (non-deterministically) to one of the pure states $\ket{n}\bra{n}$ with probability $\bra{n}\rho_0\ket{n}$. The control goal is to make this convergence deterministic toward a chosen $\bar n\in\{1,\ldots,d\}$ playing the role of controller set point.  We build on the ideas in~\cite{dotsenko-et-al:PRA09,amini-et-al:ieee12,amini-et-al:cdc2011,somaraju-et-al:RevMathPh2012} to design a controller that is based on a strict Lyapunov function for the target state. In this paper we assume arbitrary controlled Kraus  operators $M^u_\mu$  that cannot be decomposed  into QND  measurement operators  $M^0_\mu$ followed by a unique controlled unitary operator $D_u$ with $D_0=\Id$ as assumed in~\cite{dotsenko-et-al:PRA09,amini-et-al:ieee12,amini-et-al:cdc2011,somaraju-et-al:RevMathPh2012} where $M^u_\mu \equiv D_u M_\mu^0$. It can be argued that the control we are proposing is non-Hamiltonian control~\cite{Romano2006,romano2006}, as the control parameter $u$ is not  necessarily  a parameter in the interaction Hamiltonian and could indeed be any parameter of an auxiliary system such as the measurement device.

To convey the main ideas involved in the control design, we begin with the case where there are no delays ($\tau=0$) and we assume the initial state is known. We wish to use the open-loop supermartingales to design a Lyapunov function for the closed-loop system. By an {open-loop supermartingale} we mean any function $V:\DD\to\RR$ such that
$\EE{V(\rho_{k+1}) |\rho_k=\rho, u_k = 0} \leq V(\rho)$
for all $\rho\in \DD$. The Lyapunov function underlying Theorem~\ref{thm:first} demonstrates how we can construct such open-loop supermartingales.
Now, at each time-step $k$, the feedback signal $u_k$ is chosen by minimizing this supermartingale $V$ knowing the state $\rho_k$:
$$u_k = \hat{u}(\rho) := \underset{u \in[-\bar u,\bar u]}{\text{argmin}} \bigg\{ \EE{V(\rho_{k+1}) |\rho_k = \rho, u_k = u}\bigg\}.$$
Here $\bar{u}$ is some small positive number that needs to be determined. Because $0\in [-\bar{u},\bar{u}]$ and the control $u_k$ is chosen to minimize $V$ at each step, we directly have that $V$ is a \emph{closed-loop supermartingale}, i.e.,
$$\tilde Q(\rho) := \EE{V(\rho_{k+1}) |\rho_k = \rho, u_k = \hat{u}(\rho)} - V(\rho) \leq 0$$
for all $\rho\in \DD$. If this supermartingale $V$ is bounded from below then we can directly apply the convergence Theorem~\ref{thm:app} in the appendix to prove that $\rho_k$ converges with probability one to the set
$I_\infty := \{\rho: \tilde Q(\rho) = 0\}.$

What remains to be done is to choose an  appropriate open-loop supermartingale $V$ so that the set $I_\infty$  is restricted to the target state $\{\ket{\bar{n}}\bra{\bar{n}}\}$. The Lyapunov function $\Gamma$ defined in~\eqref{eq:V} does not discriminate between the different basis vectors. We therefore use the Lyapunov function
$
V_\epsilon(\rho) = V_0(\rho) - \frac{\epsilon}{2}\sum_{n=1}^{ d} \bra{n}\rho\ket{n}^2
$
where
$V_0(\rho) = \sum_{n=1}^{ d} \sigma_n \bra{n}\rho\ket{n}$
and $\epsilon$ is a small positive constant. The weights $\sigma_n$ are strictly positive numbers except for $\sigma_{\bar n}=0$. This function $V_\epsilon$ is clearly a concave function of the open-loop martingales $\bra{n}\rho\ket{n}$ and  therefore is an open-loop supermartingale. Moreover, the weights $\sigma_n$ can be used to quantify the distance of the state $\rho$ from the target state $\ket{\bar n}\bra{\bar n}$ (c.f. Figure~\ref{fig:distance} below which shows how $\sigma_n$ are chosen for the experimental setting).

A state $\rho$ is in the set $I_\infty$ if and only if for all $u\in [-\bar{u},\bar{u}]$, we have
\begin{equation}\label{eqn:limSet}
\EE{V_\epsilon(\rho_{k+1}) |\rho_k = \rho, u_k = u} - V_\epsilon(\rho) \geq   0.
\end{equation}
Also from the fact that $V_\epsilon$ is an open-loop supermartingale, we have  for all $\rho\in \DD$
\begin{equation}\label{eqn:zeroFB}
\EE{V_\epsilon(\rho_{k+1}) |\rho_k = \rho, u_k = 0} - V_\epsilon(\rho) \leq 0.
\end{equation}
We prove in Lemma~\ref{lem:first} below that given any $\bar{n}\in \{1,\ldots,d\}$, we can always choose the weights $\sigma_1,\ldots,\sigma_d$ so that $V_\epsilon$ satisfies the following property:  $\forall n\in \{1,\ldots,d\}$,
$u \mapsto \EE{V_\epsilon(\rho_{k+1}) |\rho_k = \ket{n}\bra{n}, u_k = u}$
has a strict local minimum at $u=0$ for  $n = \bar{n}$ and strict local maxima at $u = 0$ for $n\neq \bar{n}$. This combined with Equation~\eqref{eqn:zeroFB} then ensures that for any  $n\neq \bar{n}$, there is some $u\in [-\bar{u},\bar{u}]$ such that
$
\EE{V_\epsilon(\rho_{k+1}) |\rho_k = \ket{n}\bra{n}, u_k = u} - V_\epsilon(\ket{n}\bra{n}) < 0.
$
Therefore using Equation~\eqref{eqn:limSet}, we know that $\ket{n}\bra{n}$ is in the limit set $I_\infty$ if and only if $n = \bar{n}$.

This idea can easily be extended to the situations where the delay $\tau$ is non zero. Take $(\rho_k,u_{k-1},\ldots,u_{k-\tau})$ as state at step $k$: denote by $\chi=(\rho,\beta_1,\ldots,\beta_\tau)$ this state where $\beta_r$ stands for the control input $u$ delayed $r$ steps. Then the state form of the delayed  dynamics~\eqref{eq:dynamic} is governed  by the following Markov chain
\begin{equation}\label{eq:chi}
  \left\{\begin{array}{l}
    \rho_{k+1}=\MM_{\mu_k}^{\beta_{\tau,k}}(\rho_k)
    \\
    \beta_{1,k+1}= u_k, \quad
   \beta_{r,k+1} = \beta_{r-1,k} \text{ for }r=2,\ldots \tau.
\end{array}\right.
\end{equation}
The goal is to design a feedback law $ u_k = f(\chi_k)$ that globally stabilizes this  Markov chain $\chi_k$ towards a chosen target state $\bar\chi=(\ket{\bar n}\bra{\bar n},0,\ldots,0)$ for some $\bar n \in\{1,\ldots, d\}$. In Theorem~\ref{thm:main}, we show how to design a feedback relying on the control Lyapunov function $W_\epsilon(\chi) = V_\epsilon(\KK^{\beta_1}(\KK^{\beta_2}(\ldots\ldots \KK^{\beta_\tau}(\rho) \ldots ) ))$. The idea is to use a predictive filter to estimate the state of the system $\tau$ time-steps later.

Finally, we address the situation where the initial state of the system is not fully known but only estimated by $\rhoe_0$. We show under some assumptions on the initial condition that, the feedback law based on the state $\rhoe_k$ of the miss-initialized filter still ensures the convergence of $\rhoe_k$ as well as the well-initialized conditional state $\rho_k$  towards $\ket{\bar{n}}\bra{\bar{n}}$. This demonstrates how the control algorithm is robust to uncertainties in the initialization of the estimated state of the quantum system.

\subsection{Choosing the weights $\sigma_n$}\label{ssec:weights}
The construction of the control Lyapunov function relies on two  lemmas.
\begin{lem}\label{lem:R}
Consider  the $d\times d$ matrix $R$  defined by
\begin{multline*}
R_{n_1,n_2}=\sum_\mu \Big(2\left|\bra{n_1}\tfrac{d {M_{\mu}^u}^\dag}{du}\big|_{u=0}\ket{n_2}\right|^2+\\2\delta_{n_1,n_2}\Re\big( c_{\mu,n_1}\bra{n_1}\tfrac{d^2 {M_{\mu}^u}^\dag}{du^2}\big|_{u=0}\ket{n_2}\big)\Big).
\end{multline*}
When $R\neq 0$, the non-negative   $P=\Id -R/\tr{R}$ is a right stochastic matrix.
\end{lem}

\begin{proof}
For $n_1\neq n_2$, $R_{n_1,n_2} \geq 0$. Thus $R$ is a Metzler matrix~\footnote{A Metzler matrix is a matrix such that all the off-diagonal components are non-negative.}. Let us prove that the sum of each row vanishes. This results from identity $\sum_\mu {M_\mu^u}^\dag M_\mu^u=\II$. Deriving twice versus $u$ the relation
\begin{multline*}
 \sum_{\mu,n_2} \bra{n_1}{M_\mu^u}^\dag \ket{n_2}\bra{n_2} M_\mu^u\ket{n_1}= \\
 \sum_\mu \bra{n_1}{M_\mu^u}^\dag M_\mu^u\ket{n_1}=1,
\end{multline*}
yields
\begin{multline*}
\sum_{\mu,n_2}
 2  \bra{n_1}\tfrac{d{M_\mu^u}^\dag}{du}\ket{n_2}\bra{n_2} \tfrac{dM_\mu^u}{du}\ket{n_1}
+\\   \bra{n_1}\tfrac{d^2{M_\mu^u}^\dag}{du^2} \ket{n_2}\bra{n_2} M_\mu^u\ket{n_1}
+\\   \bra{n_1}{M_\mu^u}^\dag \ket{n_2}\bra{n_2}  \tfrac{d^2M_\mu^u}{du^2}\ket{n_1}
=0
.
\end{multline*}
Since for $u=0$, $\bra{n_2} M_\mu^0\ket{n_1}= \delta_{n_1,n_2} c_{\mu,n_1}$ the above sum corresponds  to
$\sum_{n_2} R_{n_1,n_2}$. Therefore, the diagonal elements of $R$ are non-positive.  If $R\neq 0$, then $\tr{R} < 0$ and the matrix  $P=\Id-R/\tr{R}$ is well defined with non-negative entries.  Since the sum of each row of $R$ vanished, the sum of each row of $P$ is equal to $1$.  Thus $P$ is a right stochastic matrix.
\end{proof}
To the Metzler matrix $R$ defined in Lemma~\ref{lem:R}, we associate its directed graph denoted by $G.$ This graph admits $d$ vertices labeled by $n\in\{1,\ldots,d\}$. To each strictly
positive off-diagonal element of the matrix $R$, say, on the $n_1$'th row and the $n_2$'th
column we associate an edge from vertex $n_1$ towards vertex $n_2.$
\begin{lem}\label{lem:first}
Assume the  directed graph $G$ of the  matrix $R$ defined in  Lemma~\ref{lem:R}  is  strongly  connected, i.e., for any $n,n^\prime\in\{1,\ldots,d\}$, $n\neq n^\prime$, there exists a chain of $r$ distinct elements $(n_j)_{j=1,\ldots,r}$ of $\{1,\ldots,d\}$ such that $n_1=n$, $n_r=n^\prime$ and for any $j=1,\ldots r-1$, $R_{n_j,n_{j+1}}\neq 0$. Take $\bar n\in\{1, \ldots, d\}$.   Then, there exist $d-1$ strictly positive real numbers $e_n$, $n\in\{1,\ldots, d\}\setminus\{\bar n\}$, such that
\begin{itemize}
\item for any   reals $\lambda_n$, $n\in\{1,\ldots, d\}\setminus\{\bar n\}$, there exists a unique  vector $\sigma=(\sigma_n)_{n\in\{1,\ldots,d\}}$ of $\RR^d$ with $\sigma_{\bar n} =0$ such that
$R\sigma=\lambda$ where $\lambda$ is the vector of $\RR^d$ of components $\lambda_n$ for $n\in\{1,\ldots, d\}\setminus\{\bar n\}$ and $\lambda_{\bar n} = - \sum_{n\neq \bar n} e_n \lambda_n$;
if additionally  $\lambda_n <0$ for all $n\in\{1,\ldots, d\}\setminus\{\bar n\}$, then $\sigma_n > 0$ for all $n\in\{1,\ldots, d\}\setminus\{\bar n\}$.

\item for any  vector $\sigma\in\RR^d$, solution of $R\sigma=\lambda\in\RR^d$, the function $V_0(\rho)=\sum_{n=1}^d\sigma_n\bra{n}\rho\ket{n}$  satisfies
 $$
 \left.\frac{d^2 V_0\big(\KK^u(\ket{n}\bra{n})\big)}{d u^2}\right|_{u=0}= \lambda_n\qquad \forall n\in\{1,\ldots,d\}.
 $$
 \end{itemize}
\end{lem}
\begin{proof}
Since the directed graph $G$ coincides with the directed graph of the right stochastic matrix $P$ defined in Lemma~\ref{lem:R}, $P$ is  irreducible. Since it is a right stochastic matrix, its spectral radius  is equal to $1$. By Perron-Frobenius theorem for non-negative irreducible matrices, this spectral radius, i.e., $1$, is also an  eigen-value of $P$ and  of $P^T$,  with  multiplicity one and  associated to  eigen-vectors having  strictly positive entries: the right eigen-vector ($P w= w$) is obviously $w=(1,\ldots,1)^T$; the left eigen-vector  $e=(e_1, \ldots, e_n)^T$ ($P^T e= e$) can be chosen such that $e_{\bar n} =1$.  Consequently, the rank of $R$ is $d-1$ with $\ker(R) = \text{span}( w)$ and $\text{Im}(R)= e ^{\perp}$ where $e ^{\perp}$ is the hyper-plane orthogonal to $e$. Since $ e^T \lambda = 0$, $\lambda\in \text{Im}(R)$,  exists $\sigma$ such that $R \sigma = \lambda$. Since $\ker(R) = \text{span}( w)$, there is a unique $\sigma$ solution of $R\sigma = \lambda$ such that $\sigma_{\bar n}=0$.
The fact that $\sigma_n > \sigma_{\bar n}$  when $\lambda_n<0$ for $n\neq \bar n$, comes from elementary manipulations of  $P\sigma = \sigma - \lambda /\tr{R}$ showing that $\min_{n\neq \bar n} \sigma_n > \sigma_{\bar n}$.

$\forall n$, set $g_n^u= V_0(\KK^u(\ket{n}\bra{n}))=\sum_{l}\sigma_l\bra{l}\KK^u(\ket{n}\bra{n})\ket{l}.$ Set $P_l^u:=\bra{l}\KK^u(\ket{n}\bra{n})\ket{l}.$ Then $
\tfrac{d P_l^u}{du}$ is given by $\sum_\mu\bra{l}\tfrac{d M_\mu^u}{du}\ket{n}\bra{n}{M_\mu^u}^\dag\ket{l}+\bra{l}M_\mu^u\ket{n}\bra{n}\tfrac{{d M_\mu^u}^\dag}{du}\ket{l}
$
and
\begin{align*}
\tfrac{d^2 P_l^u}{du^2}|_{u=0}&=\sum_\mu\Big(\bra{l}\tfrac{d^2 M_\mu^u}{du^2}\big|_{u=0}\ket{n}\bra{n}{M_\mu^0}^\dag\ket{l}
\\&+\bra{l}\tfrac{d M_\mu^u}{du}\big|_{u=0}\ket{n}\bra{n}\tfrac{d{M_\mu^u}^\dag}{du}\big|_{u=0}\ket{l}\\&+\bra{l}M_\mu^0\ket{n}\bra{n}\tfrac{{d^2 M_\mu^u}^\dag}{du^2}\big|_{u=0}\ket{l}\Big)\\
&=
\sum_\mu 2\Big(\left|\bra{n}\tfrac{d{M_\mu^u}^\dag}{du}\big|_{u=0}\ket{l}\right|^2\\&+2\delta_{nl}\Re \big(c_{\mu,n}\bra{n}\tfrac{{d^2 M_\mu^u}^\dag}{du^2}\big|_{u=0}\ket{l}\big)\Big).
\end{align*}
Therefore $
\tfrac{d^2 P_l^u}{du^2}|_{u=0}=R_{n,l}$ and $
\left.\tfrac{d^2 V_0\big( \KK^u(\ket{n}\bra{n})\big)}{d u^2}\right|_{u=0}=\sum_{l=1}^{ d} R_{n,l}\sigma_l=\lambda_n.$
\end{proof}
\subsection{The global stabilizing feedback}\label{ssec:stab}
The main result of this section is expressed through the following theorem.
\begin{theorem}\label{thm:main} Consider the Markov chain~\eqref{eq:chi} with Assumptions~\ref{assum:imp},~\ref{assum:impt} and~\ref{assum:imo}. Take $\bar n\in\{1,\ldots,d\}$ and assume that the directed graph $G$ associated to the Metzler  matrix $R$ of Lemma~\ref{lem:R} is strongly connected. Take $\epsilon >0$,   $\sigma\in\RR_+^d$ the solution of $R\sigma=\lambda$  with $\sigma_{\bar n}=0$,   $\lambda_n<0$ for $n\in\{1,\ldots,d\}\setminus\{\bar n\}$,   $\lambda_{\bar n} = -\sum_{n\neq \bar n} e_n \lambda_n$ (see Lemma~\ref{lem:first}) and
$
V_\epsilon(\rho)= \sum_{n=1}^d \sigma_n \bra n \rho \ket n - \tfrac{\epsilon}{2} (\bra n \rho \ket n)^2
.
$
Take $\bar u>0$ and consider the following feedback law
$$
u_k=\underset{\xi\in[-\bar u,\bar u]}{\text{argmin}}\big(\EE {W_\epsilon(\chi_{k+1})|\chi_k,u_k=\xi}\big)
:=f(\chi_k)
$$
where
$
W_\epsilon(\chi) = V_\epsilon(\KK^{\beta_1}(\KK^{\beta_2}(\ldots\ldots \KK^{\beta_\tau}(\rho) \ldots ) ))
.
$
Then there exists $u^*>0$ such that, for all $\bar u \in]0,u^*]$ and  $\epsilon\in\left]0,\min_{n\neq \bar n} \left(\frac{\lambda_{n}}{R_{n,n}}\right)\right]$,     the closed-loop Markov chain of state $\chi_{k}$ with the feedback law of Theorem~\ref{thm:main} converges almost surely towards $(\ket{\bar n} \bra{\bar n},0,\ldots,0)$  for any initial condition $\chi_0\in\DD \times[-\bar u,\bar u]^\tau$.
\end{theorem}
\begin{proof}
For the sake of simplicity, first we demonstrate this Theorem for $\tau=1$ and thus for $\chi=(\rho,\beta_1)$. We then explain
how this proof may be extended to arbitrary $\tau > 1$. Here,
$
\EE{W_\epsilon(\chi_{k+1})|\chi_k,u_k}$ is given by $\sum_{\mu} p_{\mu,\rho_k}^{\beta_{1,k}}V_\epsilon(\KK^{u_k}(\MM_\mu^{\beta_{1,k}}(\rho_k)))
$ that can also be presented as $
\sum_{\mu} p_{\mu,\rho_k}^{\beta_{1,k}}V_\epsilon(\KK^{0}(\MM_\mu^{\beta_{1,k}}(\rho_k)))-Q_2(\chi_k)
$
with  $Q_2\geq 0$:
\begin{multline}\label{eq:Q2}
    Q_2(\chi_k)=
    \sum_{\mu} p_{\mu,\rho_k}^{\beta_{1,k}}V_\epsilon(\KK^{0}(\MM_\mu^{\beta_{1,k}}(\rho_k)))
    -\\
    \underset{\xi\in[-\bar u,\bar u]}{\text{min}}\sum_{\mu} p_{\mu,\rho_k}^{\beta_{1,k}}V_\epsilon(\KK^{\xi}(\MM_\mu^{\beta_{1,k}}(\rho_k)))
    .
\end{multline}
Since  $V_\epsilon(\KK^{0}(\rho))=V_{\epsilon}(\rho)$ for any $\rho\in\DD $, we have
\begin{equation}\label{eq:EW}
    \EE{W_\epsilon(\chi_{k+1})|\chi_k,u_k} = W_\epsilon(\chi_k) - Q_1(\chi_k) - Q_2(\chi_k)
\end{equation}
where $Q_1$ is defined in~\eqref{eq:Q1}.
Thus, $W_\epsilon(\chi_k)$ is a supermartingale. According to Theorem~\ref{thm:app}, the $\omega$-limit set $I_\infty$ of $\chi_k$ is included in
$$
\left\{(\rho,\beta_1)\in\DD \times[-\bar u,\bar u] ~|~ Q_1(\rho,\beta_1)=0,~Q_2(\rho,\beta_1)=0\right\}.
$$
We  show in 3 steps that $I_\infty$ is restricted to the target state. \\
{\bf Step 1.} For all $\delta > 0$, there exists $ \bar{u}>0$ such that
$I_\infty\subset \bigcup_{n=1}^d \bigg\{(\rho,\beta_1)\in\DD \times[-\bar u,\bar u] ~|~ \bket{n|\rho|n} \geq 1-\delta\bigg\}.$\\
{\bf Proof of step 1.} Step 1 follows from $Q_1(\chi) = 0$ for all $(\rho,\beta_1) \in I_\infty$ and technical Lemma~\ref{lem:Q1} of  the appendix.\\
{\bf Step 2.} There exist  $\delta' > 0$ and $\bar{u}> 0$ such that
$I_\infty\subset \{(\rho,\beta_1)\in\DD \times[-\bar u,\bar u] ~|~ \bket{\bar{n}|\rho|\bar{n}} \geq 1-\delta'\}.$\\
{\bf Proof of step 2.}
If $Q_2(\rho,\beta_1)=0$, the minimum of the function
$
[-\bar u,\bar u]\ni \xi \mapsto F_{\rho,\beta_1}(\xi)=\sum_{\mu} p_{\mu,\rho}^{\beta_1}V_\epsilon(\KK^{\xi}(\MM_\mu^{\beta_1}(\rho)))
$
is attained at $\xi=0$. By construction of $V_\epsilon$, we know that, for $\beta_1=0$,  $\rho=\ket{n}\bra{n}$,
$\left.\frac{d^2F_{\ket{n}\bra{n},0}(\xi)}{d\xi^2}\right|_{\xi=0} <0$
if $n\neq \bar n$, and $\epsilon\leq \min_{n\neq \bar n} \left(\frac{\lambda_{n}}{R_{n,n}}\right).$
The dependence of $d^2F_{\rho,\beta_1}/d\xi^2$ versus $\rho$ and $\beta_1$ is continuous
from Assumption~\ref{assum:imo}. Therefore,  $\exists \delta',\bar{u} >0$  such that for all $(\rho,\beta_1)$ in
$ \bigcup_{n\neq \bar{n}} \{(\rho,\beta_1)\in\DD \times[-\bar u,\bar u] ~|~ \bket{n|\rho|n} \geq 1-\delta'\}$
we have $\left.\frac{d^2F_{\rho,\beta_1}(\xi)}{d\xi^2}\right|_{\xi=0} <0.$ Therefore, such  $(\rho,\beta_1)$ and
$\xi=0$, cannot minimize $F_{\rho,\beta_1}(\xi)$ and we have
$
I_\infty\cap \bigcup_{n\neq \bar{n}} \{(\rho,\beta_1)\in\DD \times[-\bar u,\bar u] ~|~ \bket{n|\rho|n} \geq 1-\delta'\} = \emptyset.
$
Therefore  we get,
$$I_\infty\subset \{(\rho,\beta_1)\in\DD \times[-\bar u,\bar u] ~|~ \bket{\bar{n}|\rho|\bar{n}} \geq 1-\delta'\},$$
by setting $\delta = \delta'$ in Step 1, for $\bar{u}$ small enough. \\
{\bf Step 3.} $\bar{u}$ can be chosen small enough so that $I_\infty = \{(\ket{\bar{n}}\bra{\bar{n}},0)\}$.\\
{\bf Proof of step 3.}
By construction of $V_\epsilon$, we have for $\rho = \ket{\bar{n}}\bra{\bar{n}}$ and $\beta_1 = 0$,
$
\left.\frac{d^2 F_{\ket{\bar{n}}\bra{\bar{n}},0}(\xi)}{d\xi^2}\right|_{\xi=0} >0.
$
By continuity of $d^2F_{\rho,\beta_1}/d\xi^2$ with respect to $\rho$ and $\beta_1$, we can choose $\bar{\delta}$ and $\bar{u}$ small enough
such that for all $\beta_1\in [-\bar{u},\bar{u}]$ and $\rho$ satisfying $\bket{\bar{n}|\rho|\bar{n}} \geq 1-\bar{\delta}$ we have
$
\left.\frac{d^2 F_{\rho,\beta_1}(\xi)}{d\xi^2}\right|_{\xi=0} > 0.
$
This implies that on the domain
$\{(\rho,\beta_1,\xi)\in\DD \times[-\bar u,\bar u]^2 ~|~ \bket{n|\rho|n} \geq 1-\bar{\delta}\},$
$F$ is  a uniformly strongly convex function of $\xi\in[-\bar u,\bar u]$.
Thus the argument of its minimum over $\xi\in[-\bar u,\bar u]$  is a continuous function of  $\rho$ and $\beta_1$.
Now we choose $\delta'' = \min\{\bar{\delta}/2,\delta'/2\}$, where $\delta'$ is as in step 2. Take a convergent subsequence of $\{\chi_k\}_{k=1}^\infty$ (that we still denote by $\chi_k$ for simplicity sakes). Its  limit $(\rho_\infty,\beta_{1,\infty})$ belongs to  $I_\infty$ and also to $\{(\rho,\beta_1)\in
\DD \times[-\bar u,\bar u] ~|~ \bket{n|\rho|n} \geq 1-\bar{\delta}\}.$
Therefore,
$u_k = \underset{\xi\in[-\bar u,\bar u]}{\text{argmin}} F_{\rho_k,\beta_{1,k}}(\xi)$ converges to
$\underset{\xi\in[-\bar u,\bar u]}{\text{argmin}} F_{\rho_\infty,\beta_{1,\infty}}(\xi)=0$ since
$Q_2(\rho_\infty,\beta_{1,\infty}) = 0$.
Because $\beta_{1,k}=u_{k-1}$, $\beta_{1,k}$ tends almost surely towards $0$.
We know that  $Q_1(\rho_k,\beta_{1,k})$ and $Q_2(\rho_k,\beta_{1,k})$ tend  also almost surely to $0$ (Theorem~\ref{thm:app}) and by uniform (in $\rho$) continuity with respect to $\beta_{1,k}$,
$Q_1(\rho_k,0)$ and $Q_2(\rho_k,0)$  tend almost surely to $0$. Since $Q_2(\rho,0)=Q_1(\rho,0)=0$ implies
$\rho=\ket{\bar  n}\bra{\bar n}$, $\rho_k$ converges almost surely towards $\ket{\bar n}\bra{\bar n}$. This completes the
proof for $\tau = 1$. \\
{\bf Extension to $\tau > 1$.}
For $\tau >1$ and $\chi=(\rho,\beta_1,\ldots,\beta_\tau)$,  the proof is very similar. We still have~\eqref{eq:EW} with $Q_1$ and $Q_2$
given by formulae analogous to~\eqref{eq:Q1} and~\eqref{eq:Q2}:
$\beta_{1,k}$ is replaced by $\beta_{\tau,k}$;  $\MM_\mu^{\beta_{1,k}}(\rho_k)$ is replaced by
$\KK^{\beta_{1,k}}\ldots\KK^{\beta_{\tau-1,k}}(\MM_\mu^{\beta_{\tau,k}}(\rho_k))$;
 $p_{\mu,\rho_k}^{\beta_{\tau,k}}$
 is replaced by
 $
 \tr{\KK^{\beta_{1,k}}\ldots\KK^{\beta_{\tau-1,k}}(M_\mu^{\beta_{\tau,k}} \rho_k {M_\mu^{\beta_{\tau,k}}}^\dag)}
 $ since Kraus maps are trace preserving. The analogue of technical Lemma~\ref{lem:Q1}, whose proof relies on similar continuity and compactness arguments, has now the following statement:
{\em
$\exists C>0$ such that $\forall (\rho,\beta_1,\ldots,\beta_\tau)\in\DD \times [-\widetilde{u},\widetilde{u}]^\tau $ satisfying
$Q_1(\rho,\beta_1,\ldots,\beta_\tau)=0$,
$\exists n\in\{1,\ldots,d\}$ such that $\rho_{n,n} \geq 1-C (|\beta_1|+|\beta_2|+\ldots+|\beta_\tau|)$.
}\\
 The last part of the proof showing that, for $\bar u$ small enough, the control  $u_k$ is a continuous function of $\chi_k$ when $\chi_k$ is in the neighborhood of the $\omega$-limit set
 $\{\chi\in\DD \times[-\bar u,\bar u]^\tau~|~Q_1(\chi)=Q_2(\chi)=0\}$ remains almost the same.
\end{proof}
\subsection{Quantum filter and robustness property of filter} \label{ssec:Qfilter}
When the measurement process is fully efficient (i.e., the detectors are ideal with detection efficiency equal to one and they observe all the quantum jumps) and the jump model~\eqref{eq:dynamic} admits no error, the Markov process~\eqref{eq:chi} represents a natural choice for estimating  the hidden state  $\rho.$ Indeed, the estimate  $\rhoe$ of $\rho$    satisfies the following recursive dynamics
\begin{equation}\label{eq:filter}
\rhoe_{k+1}=\MM_{\mu_k}^{u_{k-\tau}}(\rhoe_k),
\end{equation}
where the measurement outcome $\mu_k$ is driven by~\eqref{eq:dynamic}. In practice,  the control $u_k$ defined in Theorem~\ref{thm:main} could only depend  on this estimation $\rhoe_k$ replacing $\rho_k$ in $\chi_k$.
If $\rhoe_0=\rho_0$, then $\rhoe_k\equiv \rho_k$ and Theorem~\ref{thm:main} ensures convergence towards the target state.
Otherwise, the following result guaranties the convergence of such observer/controller scheme when $\ker(\rhoe_0)\subset\ker(\rho_0)$.
\begin{theorem}\label{thm:obscontroller}
Consider the recursive Equation~\eqref{eq:dynamic} and assume that the assumptions of Theorem~\ref{thm:main} are satisfied. For each measurement outcome $\mu_k$  given by~\eqref{eq:dynamic}, consider the estimation $\rhoe_{k}$ given by~\eqref{eq:filter} with an initial condition $\rhoe_0$. Set
$u_k=f(\rhoe_k,u_{k-1},\ldots,u_{k-\tau})$ where $f$ is given in Theorem~\ref{thm:main}.
Then there exists $u^*>0$ such that, for all $\bar u \in]0,u^*]$ and  $\epsilon\in\left]0,\min_{n\neq \bar n} \left(\frac{\lambda_{n}}{R_{n,n}}\right)\right]$,
$\rho_k$ and $\rhoe_k$  converge almost surely towards the target state $\ket{\bar n}\bra{\bar n}$ as soon as  $\ker(\rhoe_0)\subset\ker(\rho_0)$.
\end{theorem}
\begin{proof} Set $\bar\rho=\ket{\bar n}\bra{\bar n}$.
The main idea is that $\EE{\tr{\rho_k\bar\rho}|\rho_0,\rhoe_0}$ (where we take the expectation over all jump realizations) depends linearly on $\rho_0$ even though we apply the feedback control.
Since
$u_{k-\tau}$ is a  function of $(\rhoe_0,\mu_0,\ldots,\mu_{k-\tau-1})$,  $
 \EE{\tr{\rho_k\bar\rho}|\rho_0,\rhoe_0}$ is given by $\sum_{\mu_0,\ldots,\mu_{k-1}}\tr{\widetilde{M}_{\mu_{k-1}}^{u_{k-\tau}}\big(\ldots
 \big(\widetilde{M}_{\mu_0}^{u_{-\tau}}(\rho_0)\big)\big)\bar\rho}$,
where $\widetilde{M}_\mu^u(\rho)=M_\mu^u\rho{M_\mu^u}^\dag$.
The linearity of
$\EE{\tr{\rho_k\bar\rho}|\rho_0,\rhoe_0}$ with respect to $\rho_0$ is thus verified.  Now we apply the assumption $\ker(\rhoe_0)\subset\ker(\rho_0)$ and therefore one can find a constant $\gamma\in]0,1[$ and a well-defined density matrix $\rho_0^c$ in $\DD ,$ such that
$\rhoe_0=\gamma\rho_0+(1-\gamma)\rho_0^c.$
Applying the dominated convergence theorem:
$
\lim_{k\rightarrow\infty} \EE{\tr{\rho_k \bar\rho}~|~\rhoe_0,\rhoe_0}=1.
$
By the linearity of $ \EE{\tr{\rho_k \bar\rho}~|~\rho_0,\rhoe_0}$  with respect to $\rho_0$, $\EE{\tr{\rho_k \bar\rho}~|~\rhoe_0,\rhoe_0}$ is given by $\gamma \EE{\tr{\rho_k \bar\rho}~|~\rho_0,\rhoe_0}+(1-\gamma)\EE{\tr{\rho_k \bar\rho}~|~\rho_0^c,\rhoe_0}.$
As $ \EE{\tr{\rho_k \bar\rho}~|~\rho_0,\rhoe_0}\leq 1 $ and $ \EE{\tr{\rho_k \bar\rho}~|~\rho^c_0,\rhoe_0}\leq 1$, we necessarily have that both of them converge to $1$ since $0<\gamma<1$:
$
\lim_{k\rightarrow\infty} \EE{\tr{\rho_k \bar\rho}~|~\rho_0,\rhoe_0}=1.
$
This implies the almost sure convergence of $\rho_k$ towards the pure state $\bar\rho$.
We have also  $\EE{\tr{\rhoe_k\bar\rho}|\rho_0,\rhoe_0}$ depending  linearly  on  $\rho_0$. Thus
$\EE{\tr{\rhoe_k \bar\rho}~|~\rhoe_0,\rhoe_0}$ is given by $\gamma \EE{\tr{\rhoe_k \bar\rho}~|~\rho_0,\rhoe_0}+(1-\gamma)\EE{\tr{\rhoe_k \bar\rho}~|~\rho_0^c,\rhoe_0}.$
Using the fact that when $\rho_0=\rhoe_0$, $\rho_k=\rhoe_k$ for all $k$, we conclude similarly that $\rhoe_k$ converges almost surely towards $\bar\rho$ even if $\rho_0$ and $\rhoe_0$ do not coincide.
\end{proof}
\section{Imperfect Measurements} \label{sec:imperfect}
We now consider the feedback control problem in the presence of classical measurement imperfections with the possibility of detection errors. This model is a direct generalization of the ones used in~\cite{Gardiner,dotsenko-et-al:PRA09,sayrin-et-al:nature2011} (see also e.g.,~\cite{bouten-handel:2007,bouten-handel:2009} for an introduction to quantum filtering).
The imperfections in the measurement process are  described by a classical probabilistic model relying on a  left stochastic matrix
$(\eta_{\mur,\mu})$,  $\mur\in\{1,\ldots,\mr\}$ and $\mu\in\{1,\ldots,m\}$: $\eta_{\mur,\mu} \geq 0$
 and for any $\mu$, $\sum_{\mur=1}^{\mr} \eta_{\mur,\mu} =1$. The integer $\mr$ corresponds to the number of imperfect outcomes
and $\eta_{\mur,\mu}$ is the probability of having the imperfect outcome  $\mur$ knowing the perfect  one $\mu$.
Set
\begin{equation}\label{eq:filt}
\hrho_{k} = \EE{\rho_{k}|\rho_0,\mur_0,\ldots,\mur_{k-1},u_{-\tau}, \ldots, u_{k-\tau-1}}.
\end{equation}
Since $\rho_k$ follows~\eqref{eq:dynamic}, $\hrho_k$  is also governed by a  recursive equation~\cite{somaraju-et-al:acc2012}:
\begin{equation}\label{eq:hatdynamic}
 \hrho_{k+1}=\LL_{\mur_k}^{u_{k-\tau}}(\hrho_k),
\end{equation}
where for each $\mur$, $\LL_{\mur}^u$ is the super-operator defined by
 $
 \hrho \mapsto \LL_{\mur}^{u}(\hrho) = \frac{\bL_{\mur}^{u}(\hrho)}{\tr{\bL_{\mur}^{u}(\hrho)}}
 $
 with $\bL_{\mur}^{u}(\hrho) = \sum_{\mu=1}^{m} \eta_{\mur,\mu} M_{\mu}^{u} \hrho {M_{\mu}^{u}}^\dag$, and where
$\mur_k$ is a random variable taking values $\mur$ in $\{1,\ldots,\mr\}$  with probability
 $
    p_{\mur,\hrho_k}^{u_{k-\tau}}= \tr{\bL_{\mur}^{u_{k-\tau}}(\hrho_k)}
    .
$
Since $\eta_{\mur,\mu}$ is a left stochastic matrix,
$
\EE{\hrho_{k+1}|\hrho_k = \rho,u_{k-\tau} = u} = \KK^u(\hrho),
$
which is precisely the Kraus map~\eqref{eq:kraus} associated with the Markov process $\rho_k$. By Assumption~\ref{assum:imp}, each pure
state $\ket{n}\bra{n}$, $n\in\{1,\ldots,d\}$ remains a fixed point of the Markov process~\eqref{eq:hatdynamic} when $u\equiv 0$.

We now consider the dynamics of the filter state $\hrho$ in the presence  of delays in the feedback control.
Similar to the case with perfect measurements, let $\hchi_k = (\hrho_k,\beta_{1,k},\ldots,\beta_{\tau,k})$ be the filter state at step $k$, where
$\beta_{l,k}= u_{k-l}$ is the feedback control at time-step $k$ delayed $l$ steps.
Then the delay dynamics are determined by the following Markov chain
\begin{equation}\label{eq:chiMeas}
  \left\{\begin{array}{l}
    \hrho_{k+1}=\LL_{\mur_k}^{\beta_{\tau,k}}(\hrho_k)
    \\
    \beta_{1,k+1}= u_k, \quad
   \beta_{r,k+1} = \beta_{r-1,k} \text{ for }r=2,\ldots \tau.
\end{array}\right.
\end{equation}
Instead of Assumption~\ref{assum:impt} we now assume
\begin{assum}~\label{assum:hatimpt}
 For all $n_1\neq n_2$ in $\{1,\ldots, d\},$ there exists $\mur\in\{1,\ldots,\mr\}$ such that
 $
 \sum_{\mu=1}^{m} \eta_{\mur,\mu} |c_{\mu,n_1}|^2 \neq \sum_{\mu=1}^{m} \eta_{\mur,\mu}  |c_{\mu,n_2}|^2.
 $
\end{assum}
Assumption~\ref{assum:hatimpt} means that there exists a $\mur$ such that the statistics  when  $u\equiv 0$  for
obtaining the measurement result $\mur$ are different for the fixed points $\ket{n_1}\bra{n_1}$  and $\ket{n_2}\bra{n_2}$.
This follows by noting that $\tr{\bL_{\mur}^0(\ket{n}\bra{n})}=\sum_{\mu=1}^{m} \eta_{\mur,\mu}|c_{\mu,n}|^2$ for $n\in\{1,\ldots, d\}$.

We now state  the analogue of Theorem~\ref{thm:main} in the case of imperfect measurements.
\begin{theorem}\label{thm:mainMeas} Consider the Markov chain~\eqref{eq:chiMeas} with Assumptions~\ref{assum:imp}, \ref{assum:imo} and~\ref{assum:hatimpt}.
Take $\bar n\in\{1,\ldots,d\}$. Assume that the directed graph of Metzler matrix $R$ of
Lemma~\ref{lem:R} is  strongly connected. Take $\epsilon >0$,   $\sigma\in\RR_+^d$ solution of $R\sigma=\lambda$ with $\sigma_{\bar n}=0$, $\lambda_n<0$ for $n\in\{1,\ldots,d\}\setminus\{\bar n\}$,
$\lambda_{\bar n} = -\sum_{n\neq \bar n} e_n \lambda_n$ (see Lemma~\ref{lem:first}) and  set
$ V_\epsilon(\hrho)= \sum_{n=1}^d \sigma_n \bra n \widehat\rho \ket n - \tfrac{\epsilon}{2} (\bra n \hrho \ket n)^2.
$
Take $\bar u>0$ and consider the following feedback law
$
u_k=\underset{\xi\in[-\bar u,\bar u]}{\text{argmin}}\big(\EE {W_\epsilon({\hchi}_{k+1})|{\hchi}_k,u_k=\xi}\big)
=f({\hchi}_k),
$
where
$
W_\epsilon({\hchi}) = V_\epsilon(\KK^{\beta_1}(\KK^{\beta_2}(\ldots\ldots \KK^{\beta_\tau}(\hrho) \ldots ) ))
.
$\\
Then there exists $u^*>0$ such that, for all $\bar u \in]0,u^*]$ and  $\epsilon\in\left]0,\min_{n\neq \bar n} \left(\frac{\lambda_{n}}{R_{n,n}}\right)\right]$,
the closed-loop Markov chain of state ${\hchi}_{k}$ with the feedback law of Theorem~\ref{thm:mainMeas} converges almost surely towards
$(\ket{\bar n} \bra{\bar n},0,\ldots,0)$  for any initial condition ${\hchi}_0\in\DD \times[-\bar u,\bar u]^\tau$.
\end{theorem}
The proof of this theorem is almost identical to that of Theorem~\ref{thm:main} with $\MM^u_\mu$ replaced by $\LL^u_{\mur}$ and $p^u_{\mu,\rho}$ replaced by $p^u_{\mur,\hrho}$
and we do not give the details of this proof.\\
For estimating  the hidden state  $\hrho$ needed for the feedback design of Theorem~\ref{thm:mainMeas}, let us consider  the estimate  $\hrhoe$ given by
\begin{equation}\label{eq:filterMeas}
\hrhoe_{k+1}=\LL_{\mur_k}^{u_{k-\tau}}(\hrhoe_k)
\end{equation}
where $\mur_k$  corresponds to the imperfect  outcome detected at step $k$. Such $\mur_k$ is correlated  to    the perfect  and hidden outcome $\mu_k$ of~\eqref{eq:dynamic}  through   the  classical stochastic process attached to~$(\eta_{\mur,\mu})$:  for each $\mu_k$,
$\mur_k$ is a random variable  to be equal to $\mur \in\{1,\ldots,\mr\}$ with probability $\eta_{\mur,\mu_k}$. In practice,  the control $u_k$ defined in Theorem~\ref{thm:mainMeas} could only depend  on this estimation $\hrhoe_k$ replacing $\hrho_k$ in $\hchi_k=(\hrho_k,u_{k-1}, \ldots,u_{k-\tau})$. The following result guaranties the convergence of the feedback scheme  when $\ker(\hrhoe_0)\subset\ker(\rho_0)$.
\begin{theorem}\label{thm:obscontrollerMeas}
Consider the recursive Equation~\eqref{eq:dynamic} and take assumptions of Theorem~\ref{thm:mainMeas}. Consider the estimation $\hrhoe_{k}$ given by~\eqref{eq:filterMeas} with an initial condition $\hrhoe_0$. Set $u_k=f(\hrhoe_k,u_{k-1},\ldots,u_{k-\tau})$ where $f$ is given by Theorem~\ref{thm:mainMeas}.
Then there exists $u^*>0$ such that, for all $\bar u \in]0,u^*]$ and  $\epsilon\in\left]0,\min_{n\neq \bar n} \left(\frac{\lambda_{n}}{R_{n,n}}\right)\right]$,  $\rho_k$ and $\hrhoe_k$ converge almost surely towards the target state $\ket{\bar n}\bra{\bar n}$  as soon as  $\ker(\hrhoe_0)\subset\ker(\rho_0)$.
\end{theorem}
\begin{proof}
  Let us first prove that $\hrho_k$  defined by~\eqref{eq:filt} converges almost surely towards $\ket{\bar n}\bra{\bar n}$. Since $\hrho_0=\rho_0$, we have $\ker(\hrhoe_0)\subset\ker(\hrho_0)$. Thus, there exist $\rho_c\in\DD$ and  $\gamma\in]0,1[$, such that,  $\hrhoe_0=\gamma\hrho_0+(1-\gamma)\hrho_0^c$. Similarly to the proof of Theorem~\ref{thm:obscontroller},
  $\EE{\tr{\hrho_k\bar\rho}|\hrho_0,\hrhoe_0}$  depends linearly on $\hrho_0$: $\EE{\tr{\hrho_k \bar\rho}~|~\hrhoe_0,\hrhoe_0}$ is given by $\gamma \EE{\tr{\hrho_k \bar\rho}~|~\hrho_0,\hrhoe_0}+(1-\gamma)\EE{\tr{\hrho_k \bar\rho}~|~\hrho_0^c,\hrhoe_0}$.
Moreover by Theorem~\ref{thm:mainMeas}, $\EE{\tr{\hrho_k \bar\rho}~|~\hrhoe_0,\hrhoe_0}$ converges almost surely towards~$1$, thus $\EE{\tr{\hrho_k \bar\rho}~|~\hrho_0,\hrhoe_0}$ and $\EE{\tr{\hrho_k \bar\rho}~|~\hrho_0^c,\hrhoe_0}$ converge also towards $1$. Consequently $\hrho_k$ converges almost surely towards $\ket{\bar n}\bra{\bar n}$.
Since  $\hrho_k$ is the  conditional expectation of $\rho_k$ knowing the past imperfect outcomes and control inputs and since  its limit   $\ket{\bar n}\bra{\bar n}$ is a pure state,   $\rho_k$ converges necessarily  towards  the same pure state almost surely.  Convergence of  $\hrhoe_k$ relies on similar arguments exploiting  the linearity of $\EE{\tr{\hrhoe_k\bar\rho}|\hrho_0,\hrhoe_0}$ versus $\hrho_0$.
\end{proof}
\section{The photon box} \label{sec:photonbox}
In this section, we give the explicit expression of the  feedback controller  which has been experimentally tested in Laboratoire Kastler-Brossel (LKB) at Ecole Normal Sup\'{e}rieure (ENS) de Paris. We briefly summarize how the control design elucidated in this paper is applied to the LKB experiment. We refer the interested reader to~\cite{sayrin:thesis,sayrin-et-al:nature2011} for more details. This feedback controller  has been obtained by the Lyapunov design discussed in previous sections.
\subsection{Experimental system}
\begin{figure}[]
        \centerline{\includegraphics[height=5cm]{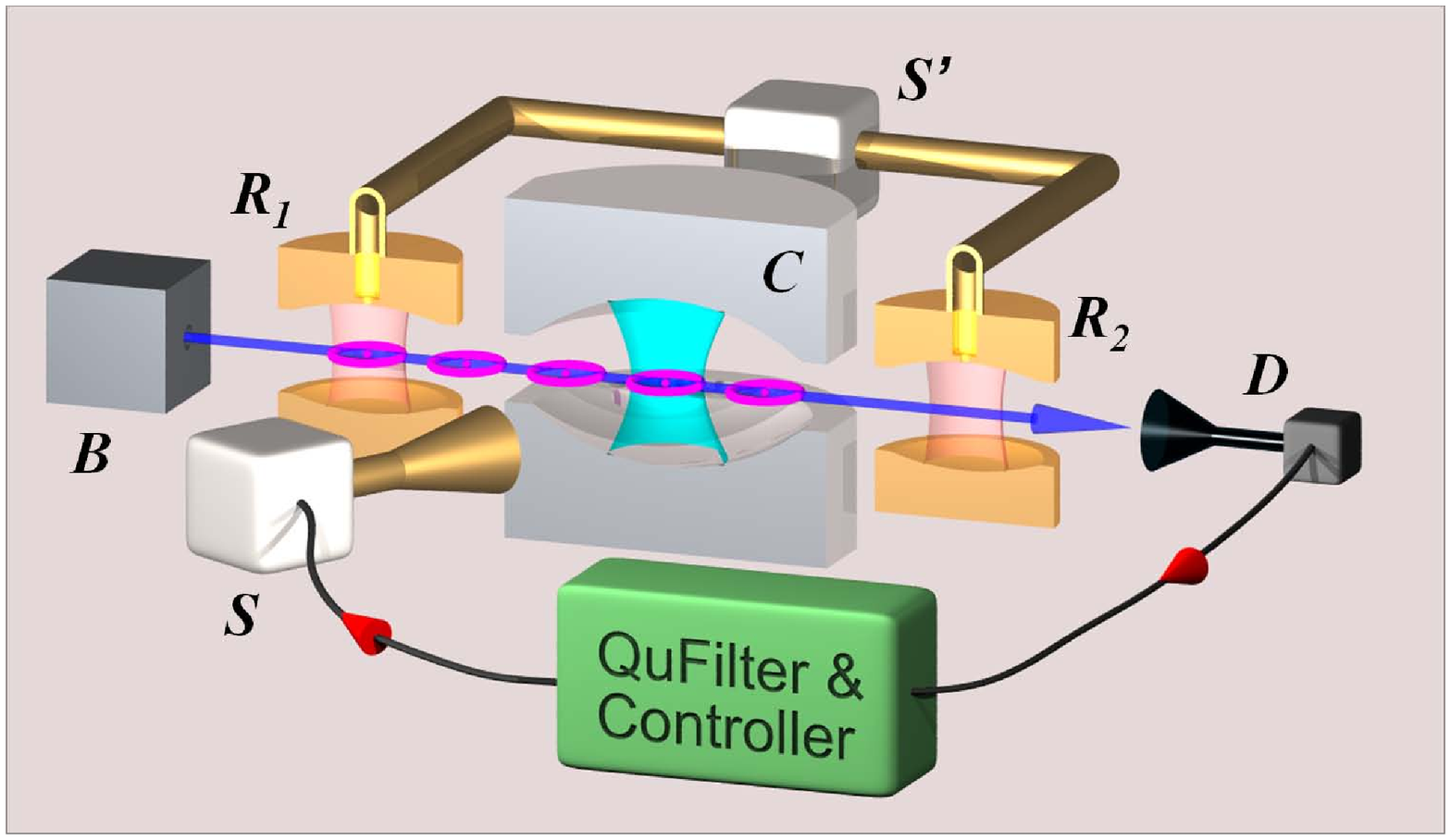}}
        \vspace{-0.2cm}
        \caption{Scheme of the experimental setup.}
        ~\label{fig:setup}
    \end{figure}

Figure~\ref{fig:setup} is a sketch of the experimental setup considered in~\cite{dotsenko-et-al:PRA09}. Two-level atoms, with state space spanned by the vectors $\ket e$ and $\ket g$ (excited and ground state respectively), act as qubits. The atoms, emitted from box B, interact individually with the electromagnetic microwave field stored in the cavity~$C$ and get entangled with it. The cavity is placed in between two additional cavities, $R_1$ and $R_2$, where the atomic state can be manipulated at will. Together, they form an atomic interferometer similar to those used in atomic clocks. The atomic state, $\ket e$ or $\ket g$, is eventually measured in the detection device~$D$.
Using the outcomes of those measurements, a quantum filter, run by a real-time computer, estimates the density matrix $\rho$ of the microwave field trapped in the cavity~$C$. It also uses all the available knowledge on the experimental setup, especially the limited detection efficiency $\varepsilon$ (percentage of atoms that are actually detected), the detection errors $\eta_e$ and $\eta_g$ ($\eta_{s}, s\in\{e,g\}$ is the probability to detect an atom in the wrong state).
A controller eventually calculates (state feedback) the amplitude $u\in\mathbb{R}$ of the classical microwave field that is injected into $C$ by an external source $S$ so as to bring the cavity field closer to the target state.
The succession of the atomic detection, the state estimation and the microwave injection define one iteration of our feedback loop. It corresponds to a sampling time of around $80~\mu$s.   For a more detailed   description see~\cite{haroche-raimond:book06,sayrin:thesis,sayrin-et-al:nature2011}.
\subsection{The controlled Markov process and quantum filter}
The state to be stabilized is the cavity state. The underlying Hilbert space is $\HH=\CC^{\nmax+1}$ which is assumed to be finite-dimensional (truncations to $\nmax$ photons). It admits $(\ket{0},\ket{1},\ldots,\ket{\nmax})$ as an ortho-normal basis.
Each basis vector $\ket{\nph}\in\CC^{\nmax+1}$ corresponds to a pure state, called photon-number state (Fock state), with precisely $\nph$ photons in the cavity, $\nph\in\{0,\ldots,\nmax\}.$ With  notations of Subsection~\ref{ssec:model}, $d=\nmax+1$ and the index $n=\nph+1$.
The three main processes that drive the evolution of our system are decoherence, injection, and measurement. Their action onto the system's state can be described by three super-operators $\superoperEvolution$, $\superoperInjection$ and $\superoperMeasurement$, respectively. We refer to~\cite{dotsenko-et-al:PRA09} for more details. All  operators are expressed in the Fock-basis  $(\ket{\nph})_{\nph=0,\ldots,\nmax}$ truncated to   $\nmax$ photons.

After taking into account our full knowledge about the experiment, we finally get the following state estimate at step $k+1$:
    \begin{equation}\label{eq:cavity}
        \hrhoe_{k+1} = \superoperMeasurement_{\mu'_k}(\superoperInjection_{u_{k-\tau}}(\superoperEvolution_\theta(\hrhoe_{k}))) \equiv \LL^{u_{k-\tau},\theta}_{\mu'_k}(\hrhoe_k),
    \end{equation}

with the following short descriptions for the super-operators $\superoperEvolution$, $\superoperInjection$ and $\superoperMeasurement.$

\begin{itemize}
\item The decoherence manifests itself through spontaneous loss or capture of a photon to or from the environment which is described as follows:
    \begin{equation}\label{eq:rho_jump}
        \rho \mapsto \superoperEvolution_\theta(\rho) = L_{0} \rho L_{0}^\dag  + L_{-} \rho L_{-}^\dag + L_{+} \rho L_{+}^\dag.
    \end{equation}
where $L_0=\Id - \theta(1/2 + \nth) \, \bN -  (\theta \nth/2) \, \Id$, $L_-=\sqrt{\theta(1+\nth)} \, \ba,$ and $L_+=\sqrt{\theta \nth} \, \ba^\dag,$
with  $\theta\ll 1,$ $\ba$ and $\ba^\dag$  photon annihilation and creation operators ($\ba \ket \nph = \sqrt{\nph}\ket{\nph\!-\!1}$ and $\ba^\dag \ket \nph = \sqrt{\nph\!+\!1} \ket{\nph\!+\!1}$) and with $\bN=\ba^\dag \ba$  the photon number operator ($\bN \ket \nph = \nph \ket \nph$). Besides, $\nth$ is the mean number of photons in the cavity mode at thermal equilibrium with its environment.
\item The evolution of the state $\rho$ after the control injection is modeled through
    \begin{equation}
        \rho \mapsto \superoperInjection_{\alphau}(\rho) =  D_\alphau \rho D_{-\alphau},
    \end{equation}
with $D_\alphau = \exp (\alphau \ba^\dag - \alphau \ba).$ In reality, the control at step $k$, $\alphau_k,$ is subject to a delay of $\tau>0$ steps which corresponds to the number of flying atoms between the cavity $C$ and the detector $D$.
 \item In the real experiment, the atom source is probabilistic and is characterized by a truncated Poisson probability distribution $P_a(n_a) \geq 0 $ to have $n_a\in\{0,1,2\}$ atom(s) in a sample (we neglect events with more than 2 atoms). This expands the set of the possible detection outcomes to $m=7$ values $\mu\in\{\o,g,e,gg,eg,ge,ee\}$, related to the following measurement operators,
    $L_{\o} = \sqrt{P_a(0)} \, \bid,$
    $L_{g}  = \sqrt{P_a(1)} \,\cos(\phi_\bN)$,
    $L_{e}  = \sqrt{P_a(1)} \,\sin(\phi_\bN)$,\\
    $L_{gg} = \sqrt{P_a(2)} \,\cos^2(\phi_\bN)$,
    $L_{ee} = \sqrt{P_a(2)} \,  \sin^2(\phi_\bN)$ and
    $L_{ge} = L_{eg}=\sqrt{P_a(2)} \, \cos(\phi_\bN) \sin(\phi_\bN)$
where $\phi_\bN =\frac{\phi_r + \phi_0({\bN}+1/2)}{2},$ $\phi_r$ and $\phi_0$ are physical parameters.

The real measurement process is not perfect: the detection efficiency is limited to $\varepsilon < 1$ and the state detection errors are non-zero ($0<\eta_{e/g}<1)$. These imperfections are taken into account by considering the left stochastic matrix $\eta_{\mur,\mu}$ which is given in~\cite{somaraju-et-al:acc2012}. Consequently, the optimal state estimate after measurement outcome $\mu'$ gets the following form:
    \begin{equation}\label{eqn:mainRes}
        \superoperMeasurement_{\mu'} (\rho) = \frac{\sum_{\mu=1}^{m} \eta_{\mu',\mu} L_{\mu} \rho L_{\mu}^\dag}
        {\tr{\sum_{\mu=1}^{m} \eta_{\mu',\mu} L_{\mu} \rho L_{\mu}^\dag}}.
    \end{equation}
The measurement operators $(L_\mu)_{1\leq\mu\leq m}$ are diagonal in the Fock basis $\{\ket \nph\}$, illustrating their quantum non-demolition nature with respect to the photon number operator and thus fulfilling Assumption~\ref{assum:imp}. Besides, Assumption~\ref{assum:hatimpt} can also be fulfilled by a proper choice of the experimental parameters $\phiR$ and $\phi_0$.

%Finally, the real measurement process is not perfect: the detection efficiency is limited to $\varepsilon < 1$ and the state detection errors are non-zero ($0<\eta_{e/g}<1)$. These imperfections are taken into account by considering the left-stochastic matrix $\eta_{\mur,\mu}$ presented in~\cite{somaraju-et-al:acc2012}. Consequently, the optimal state estimate after measurement outcome $\mu'$ gets the form
%    \begin{equation}\label{eqn:mainRes}
%        \superoperMeasurement_{\mu'} (\rho) = \frac{\sum_{\mu=1}^{m} \eta_{\mu',\mu} L_{\mu} \rho L_{\mu}^\dag}
%        {\tr{\sum_{\mu=1}^{m} \eta_{\mu',\mu} L_{\mu} \rho L_{\mu}^\dag}}.
%    \end{equation}
\end{itemize}

\subsection{Feedback controller}
For $\theta =0$ (no  cavity decoherence) the Markov model of density matrix $\hrho$  associated to  the filter~\eqref{eq:cavity} is exactly of the form~\eqref{eq:chiMeas} with $(\ket{\nph}\bra{\nph})_{\nph=0, \ldots, \nmax}$ being fixed-points in open-loop. Similarly, the underlying Markov process of the true cavity state $\rho$, which is unobservable in practice because of detection errors and delays, admits the same fixed points in open-loop.  With  parameters given in Subsection~\ref{ssec:exp} (except $\theta=0$), these Markov processes   satisfy  Assumptions~\ref{assum:imp}-\ref{assum:impt}-\ref{assum:imo} for $\rho$ and Assumptions~\ref{assum:imp}-\ref{assum:imo}-\ref{assum:hatimpt} for $\hrho$.    Moreover the Metzler matrix $R$ of Lemma~\ref{lem:first} is irreducible. Consequently the assumptions of Theorem~\ref{thm:mainMeas} are satisfied. The feedback law proposed in Theorem~\ref{thm:obscontrollerMeas} and relying on the filter state $\hrhoe$ will stabilize globally the unobservable state $\rho$ towards the  target  photon-number state $\ket{\nt}\bra{\nt}$. Numerous closed-loop simulations show that  taking $\epsilon=0$ in the  feedback law   does not  destroy  stability  and does not affect  the  convergence rates. This explains why in the simulations and experiments, we set $\epsilon=0$ despite the fact that Theorem~\ref{thm:obscontrollerMeas} guaranties convergence  only for  arbitrary small but strictly positive $\epsilon$.

%for $\theta$ positive  and small, the $\ket{\nph}\bra{\nph}$'s are no more fixed-points for $\rho$ and $\hrho$ in open-loop. nevertheless, closed-loop simulations and experimental data  of figures~\ref{fig:simulation} and~\ref{fig:experimental} indicate that a slight adaptation of the feedback scheme of theorem~\ref{thm:obscontrollermeas}  steers and maintains  $\rho$ and $\hrho$  close to the goal $\ket{\nt}\bra{\nt}$ between photon-number jumps  induced by cavity decoherence (jump operators $l_{+}$ and $l_{-}$ given in~\eqref{eq:rho_jump_componets}). the feedback of theorem~\ref{thm:obscontrollermeas} appears to be  robust enough to compensate such cavity decoherence jumps induced by a finite lifetime of the photons.

For $\theta$ positive  and small, the $\ket{\nph}\bra{\nph}$s are no more fixed-points for $\rho$ and $\hrho$ in open-loop.
Let us detail how to  adapt  the feedback scheme of Theorem~\ref{thm:obscontrollerMeas}. At each step of an ideal experiment the  control $u_k$ minimizes the Lyapunov function $V_0(\tilde\rho_k)=\sum_{\nph} \sigma_{\nph} \bra \nph \tilde\rho_k \ket \nph$ ($\epsilon$ is set to zero) calculated for state $\tilde\rho_k = \superoperInjection_{u_k}\!(\rho_k)$. In our real experiment however, we also take into account decoherence and $\tau$ flying not-yet-detected samples and therefore choose $u_k$ to minimize $V_0(\tilde\rho_k)$ for $$\tilde\rho_k = \superoperInjection_{u_k}(\superoperEvolution_\theta(\KK^{u_{k-1},\theta}(\KK^{u_{k-2},\theta}(\ldots \KK^{u_{k-\tau},\theta}(\hrhoe_k)\ldots))))$$ with $\hrhoe$ given by~\eqref{eq:cavity}.
Here we have introduced the Kraus map of the real experiment
$\KK^{u,\theta}(\rho)=  \sum_{\mu=1}^m  L_\mu (\superoperInjection_u(\superoperEvolution_\theta(\rho))) {L_\mu}^\dag$.
The control $u$ minimizing $V_0$ is approximated by
$\underset{\xi\in[-\bar u,\bar u]}{\operatorname{argmax}}\left(a_1 \xi + a_2 \xi^2/2\right)$
with $a_1 = \Tr{[\ba^\dag\!-\!\ba,\sigma_{\bN}]\rho}$ and $a_2 = \Tr{[[\ba^\dag\!-\!\ba,\sigma_{\bN}],\ba^\dag\!-\!\ba]\rho}$, where $\sigma_{\bN}$ is the diagonal operator $\sum_{\nph} \sigma_{\nph} \ket{\nph}\bra{\nph}$. The coefficients $\sigma_{\nph}$ are computed using Lemma~\ref{lem:first} where, for $\nph\neq \nt$,  $\lambda_{\nph}$ are chosen negative  and  with a decreasing  modulus versus $\nph$ in order to compensate cavity decay. For $\nt=3$, we have compared in simulations different setting  and selected the profile displayed in Figure~\ref{fig:distance}.
\begin{figure}[!t]
        \centerline{\includegraphics [height=6cm]{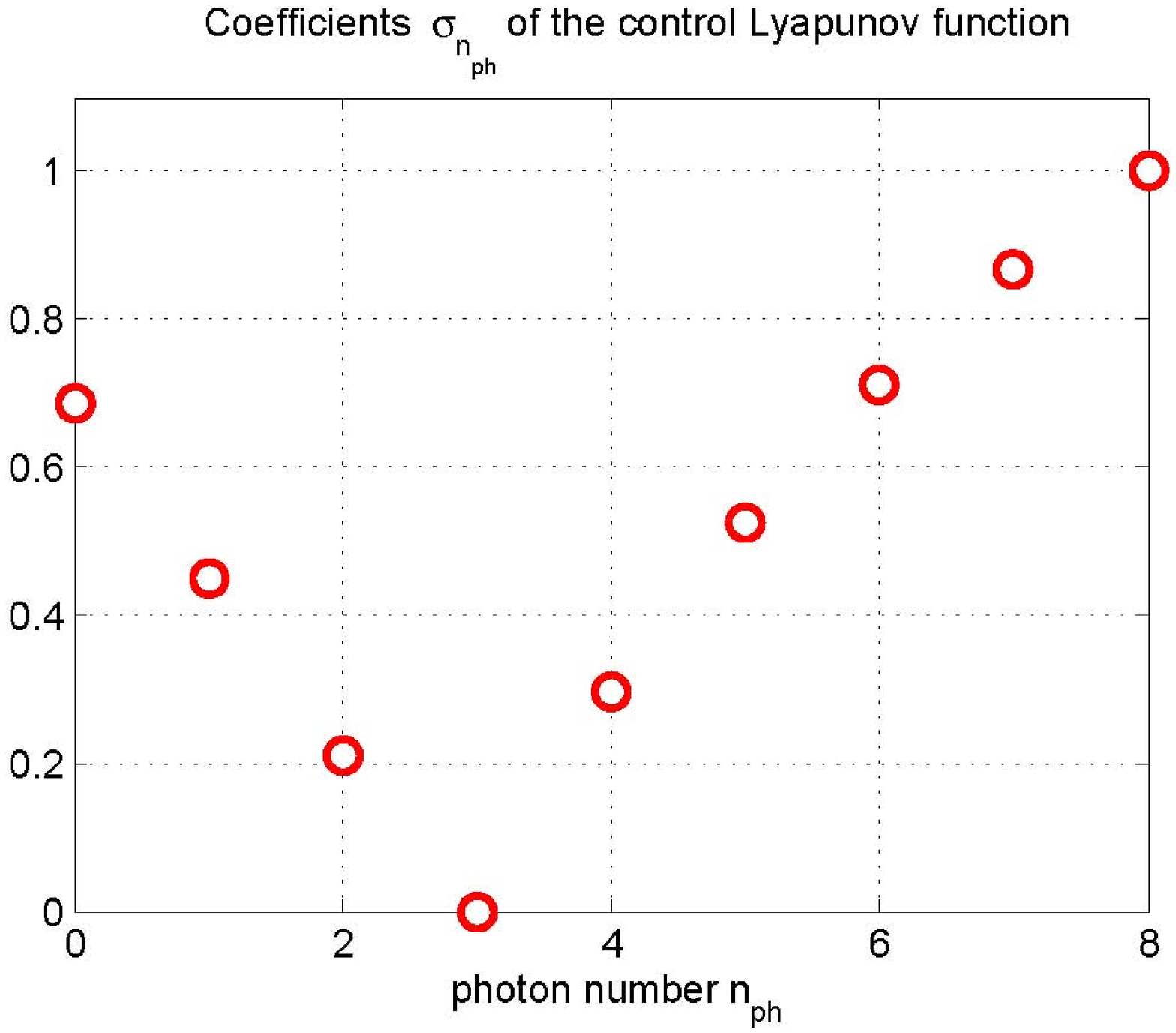}}
        \vspace{-0.2cm}
        \caption{Coefficients of the Lyapunov function $V_0(\rho)$ used in the feedback for  the simulation of Figure~\ref{fig:simulation} and for the experiment of Figure~\ref{fig:experimental}.}
        \label{fig:distance}
    \end{figure}

\subsection{Simulations and experimental results}\label{ssec:exp}
Closed-loop simulation of Figure~\ref{fig:simulation} shows a typical Monte-Carlo trajectory of the feedback loop aiming to  stabilize the 3-photon state $\ket {\nt=3}$. The experimental parameters used in the simulations are  the following: $\nmax=8$, $\phi_0 = 0.245\,\pi$, $\phi_r = \pi/2 - \phi_0(\nt + 1/2)$, $\langle n_{\textrm{a}} \rangle = 0.6$, $\varepsilon=0.35$, $\eta_e=0.13$, $\eta_g=0.11$, $\theta=0.014$,  $\nth = 0.05$, and $\tau=4$. For the feedback, $\sigma_{\nph}$ are given in Figure~\ref{fig:distance}, $\epsilon=0$ and $\bar u=1/10$.  The initial states $\rho_0$ and $\hrhoe_0$ take the following form: $\superoperInjection_{\sqrt{\nt}} (\ket{0}\bra{0})$.
\begin{figure}[!htp]
       \centerline{\includegraphics [width=0.5\textwidth]{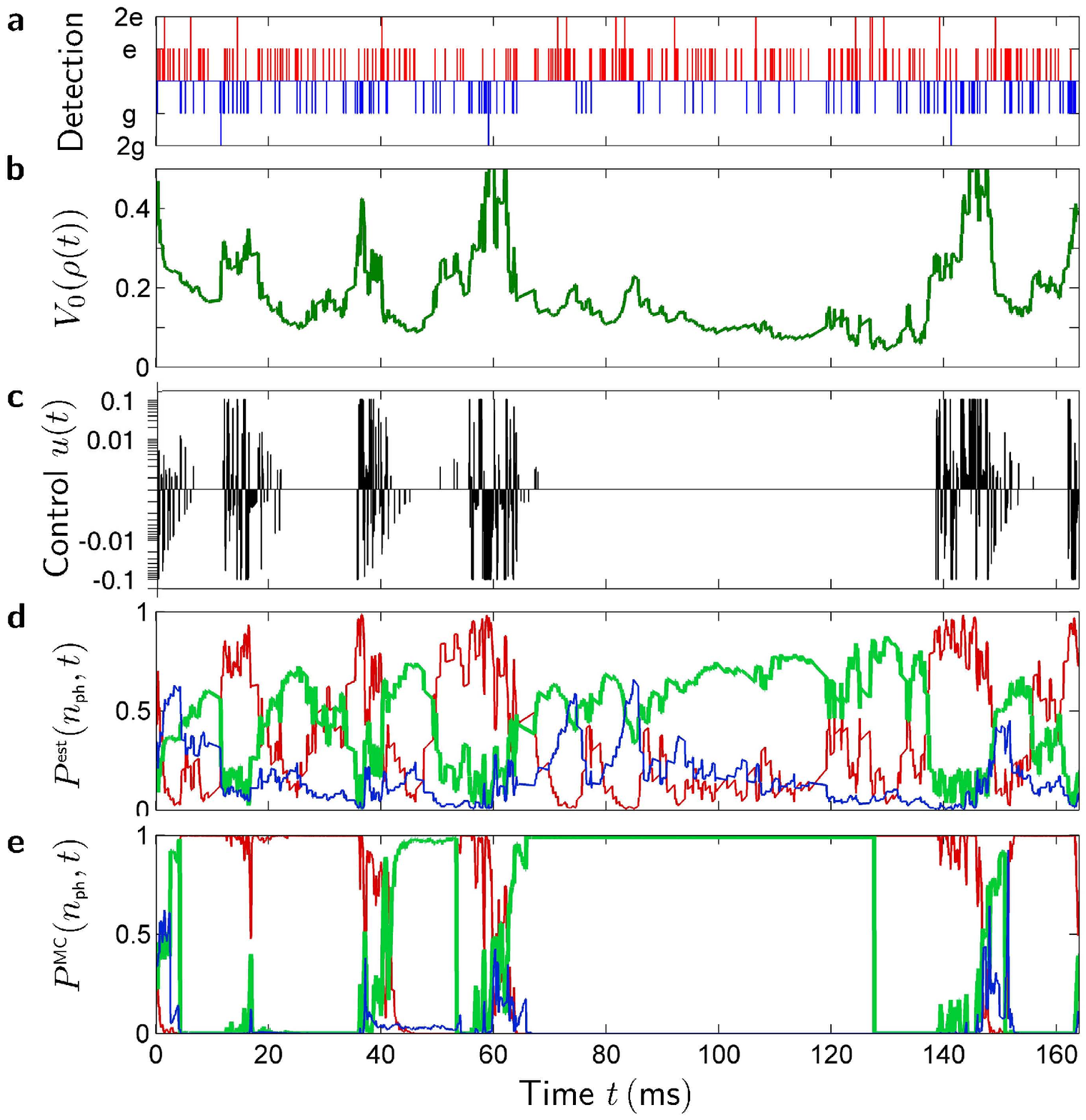}}
        \vspace{-0.2cm}
        \caption{Simulation of one Monte-Carlo trajectory in closed-loop with the target state $\ket {\nt=3}$. (a) Detection results. (b) Control Lyapunov function. (c) Control input. (d) Estimated photon number probabilities: $\sum_{\nph< \nt} \bra{\nph}\hrhoe\ket{\nph}$ in   red, $\bra{\nt}\hrhoe\ket{\nt}$ in  thick green, and $\sum_{\nph> \nt} \bra{\nph}\hrhoe\ket{\nph}$ in   blue. (e) Cavity  photon number probabilities: $\sum_{\nph< \nt} \bra{\nph}\rho\ket{\nph}$ in red, $\bra{\nt}\rho\ket{\nt}$ in thick green, and $\sum_{\nph> \nt} \bra{\nph}\rho\ket{\nph}$ in  blue. }
        \label{fig:simulation}
    \end{figure}

The results of the experimental implementation of the feedback scheme are presented in Figure~\ref{fig:experimental}. Figure~\ref{fig:experimental}(e) shows that the average fidelity of the target state is about $47\%$. Besides, the asymmetry between the distributions for $\nph<\nt$ and $\nph>\nt$ indicates the presence of quantum jumps occurring preferentially downwards ($\nt \rightarrow \nt-1)$.  Contrarily to the simulations of Figure~\ref{fig:simulation},   the cavity photon number probabilities  relying on $\rho$  are not accessible in the experimental data of Figure~\ref{fig:experimental} since we do not have access to the detection errors and to the cavity decoherence jumps~\cite{guerlin2007,Brune2008,Wang2008}. Nevertheless,  green curves  in  simulations of Figures~\ref{fig:simulation}(d) and~\ref{fig:simulation}(e) indicate  that when $\bra{\nt}\hrhoe\ket{\nt}$  exceeds  $8/10$,  $\rho$ coincides, with high probability, with $\ket{\nt}\bra{\nt}$.
\begin{figure}[!htp]
       \centerline{\includegraphics [width=0.5\textwidth]{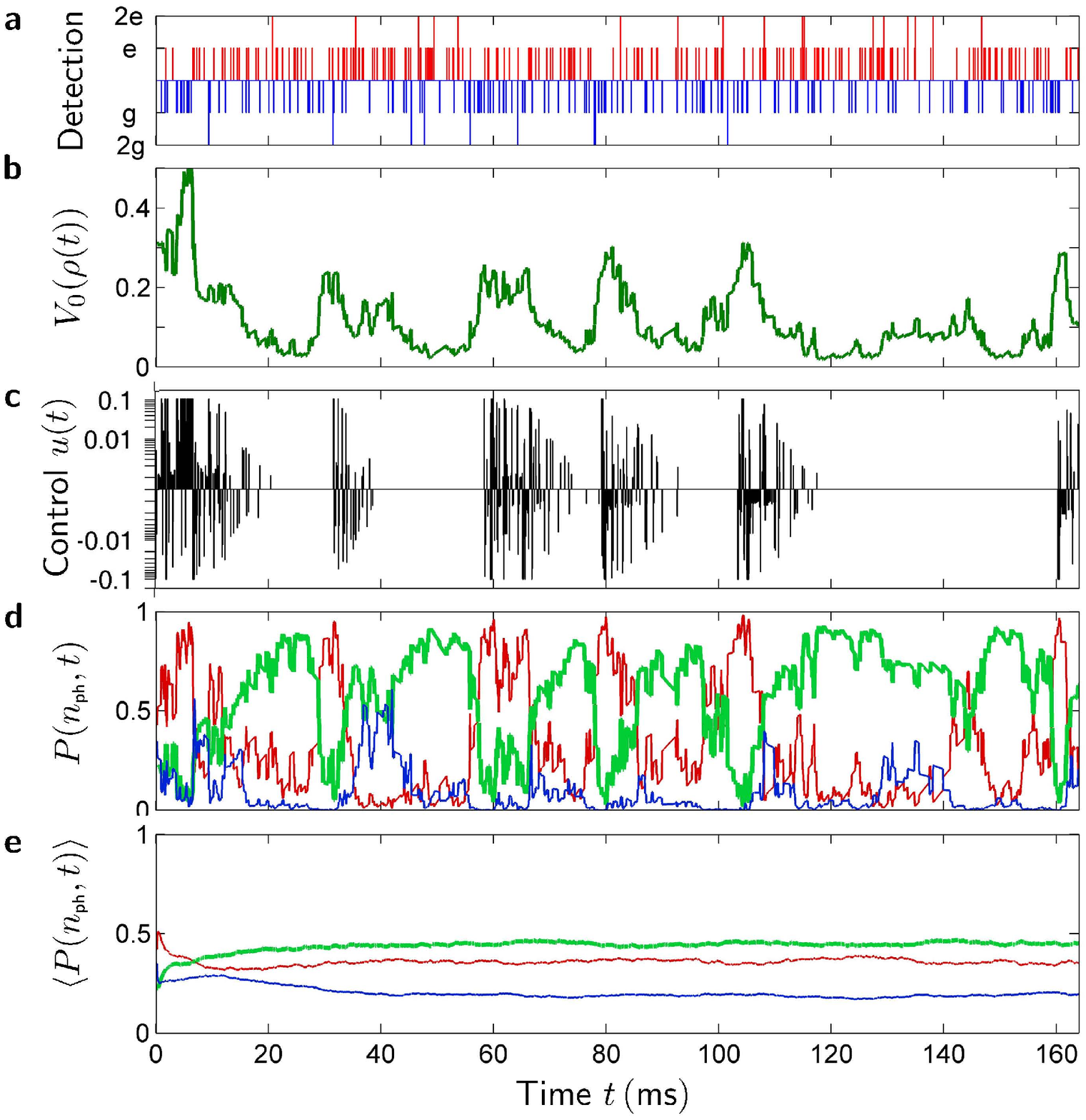}}
        \vspace{-0.2cm}
        \caption{Single experimental trajectory of the feedback loop with the target state $\ket {\nt=3}$. (a) Detection results. (b) Evolution of the control function. (c) Control injection. (d) Estimated photon number probabilities: $\sum_{\nph< \nt} \bra{\nph}\hrhoe\ket{\nph}$ in  red, $\bra{\nt}\hrhoe\ket{\nt}$ in thick green, and $\sum_{\nph> \nt} \bra{\nph}\hrhoe\ket{\nph}$ in  blue.  (e) Estimated photon number probabilities averaged over 4000 experimental  trajectories of the feedback loop.}
        \label{fig:experimental}
    \end{figure}

\section{Conclusion}
We have proposed a Lyapunov  design  for state-feedback stabilization of a discrete-time finite-dimensional quantum system  with  QND measurements. Extensions of this design  are possible in different directions such as
\begin{itemize}
   \item replacing  the continuous  and one-dimensional   input $u$ by a multi-dimensional one $(u_1,\ldots,u_p)$;
   \item assuming that $u$  belongs to a finite set of discrete values;
   \item taking an infinite dimensional state space as in~\cite{somaraju-et-al:RevMathPh2012} where the truncation to finite photon numbers is removed;
  \item  considering continuous-time systems similar to the ones  investigated in~\cite{mirrahimi-handel:siam07};
  \item  ensuring  convergence towards a sub-space~\cite{bolognani-ticozzi:ieee10}  instead of a pure-state and thus achieving   a goal  similar to error correction code as already proposed in~\cite{ahn-et-al:PRA02}.
\end{itemize}
\noindent{\bf Acknowledgements:} the authors thank Michel Brune, Serge Haroche and Jean-Michel Raimond  for enlightening discussions and advices.

\bibliographystyle{plain}

\appendix
\section{Appendix}
The following theorem is just an application of Theorem $1$ in~\cite[Ch. 8]{kushner-71}.
\begin{theorem}\label{thm:app}
\rm Let $X_k$ be a Markov chain on the compact state space $S.$ Suppose, there exists a continuous function $V(X)$ satisfying
\begin{equation}\label{eq:most}
\EE{V(X_{k+1})|X_k}-V(X_k)= -Q(X_k),
\end{equation}
where  $Q(X)$ is a non-negative continuous function of $X,$ then the $\omega$-limit set $\Omega$ (in the sense of almost sure convergence) of $X_k$ is contained by the following set
$
I_\infty=\{X|\quad Q(X)=0\}.
$
\end{theorem}

\begin{lem} \label{lem:Q1} Consider the function $Q_1$ defined by~\eqref{eq:Q1} and  $ \widetilde{u}>0$. Then there
exists $C>0$ such that for all $(\rho,\beta_1)\in\DD \times [-\widetilde{u},\widetilde{u}] $ satisfying $Q_1(\rho,\beta_1)=0$, there
exists $n\in\{1,\ldots,d\}$ such that $\rho_{n,n}=\bra n \rho \ket n  \geq 1-C |\beta_1|$.
\end{lem}
The proof is the following.
For all $n,\mu,\nu$, condition $Q_1=0$ implies that
$p_{\nu,\rho}^{\beta_1} \bra{n}M_{\mu}^{\beta_1}\rho {M_{\mu}^{\beta_1}}^\dag\ket{n}=p_{\mu,\rho}^{\beta_1}\bra{n}M_{\nu}^{\beta_1}\rho {M_{\nu}^{\beta_1}}^\dag\ket{n}$ .
Taking the sum over all $\nu$, we get
$
\bra{n}M_{\mu}^{\beta_1}\rho {M_{\mu}^{\beta_1}}^\dag\ket{n}=p_{\mu,\rho}^{\beta_1}\bra{n}\KK^{\beta_1}(\rho)\ket{n}
$ for all $n$ and $\mu$.
Since $M_\mu^{\beta_1}$ and $\KK^{\beta_1}$ are $C^2$-function of $\beta_1$,  these relations read
$
|c_{\mu,n}|^2 \rho_{n,n} = \left( \sum_{n'} |c_{\mu,n'}|^2 \rho_{n',n'} \right) \rho_{n,n} + {\beta_1}~ b_{\mu,n}(\rho,{\beta_1})
$
where $\rho_{n_1,n_2}$ stands for $\bra{n_1}\rho \ket{n_2}$ and  the scalar functions $b_{\mu,n}$ depend continuously on $\rho$ and ${\beta_1}$.\\
Let us  finish the proof  by contradiction. Assume that for all $C>0$, there exists $(\rho^C,\beta_1^C)\in\DD\times [-\widetilde{u},\widetilde{u}] $ satisfying $Q_1(\rho^C,\beta_1^C)=0$, such that
$\forall n\in\{1,\ldots,d\}$, $\rho^C_{n,n} \leq  1-C |\beta_1^C|$.
Take $C$ tending towards $+\infty$. Since $\rho^C$ and $\beta_1^C$ remain in a compact set, we can assume, up to some extraction process, that $\rho^C$ and $\beta_1^C$ converge towards $\rho^*$ and $\beta_1^*$ in $\DD$ and $[-\widetilde{u},\widetilde{u}]$. Since
$|\beta_1^C|\leq (1-\rho^C_{n,n})/C \leq 1/C$, we have  $\beta_1^*=0$.
Since
\begin{multline}\label{eq:rhoCn}
|c_{\mu,n}|^2 \rho^C_{n,n} = \rho^C_{n,n} \sum_{n'} |c_{\mu,n'}|^2 \rho^C_{n',n'}\\ +\beta_1^C~ b_{\mu,n}(\rho^C,\beta_1^C)
\end{multline}
we have by continuity for $C$ tending to  $+\infty$ and for all $n$ and $\mu$:
$
|c_{\mu,n}|^2 \rho^*_{n,n} = \left( \sum_{n'} |c_{\mu,n'}|^2 \rho^*_{n',n'} \right) \rho^*_{n,n}
$.
Thus there exists $n^*\in\{1,\ldots,d\}$ such that $\rho^* =\ket{n^*}\bra{n^*}$ (see the proof of Theorem~\ref{thm:first}).
Since $\rho^*_{n^*,n^*}=1$, for $C$ large enough, $\rho^C_{n^*,n^*} >1/2$ and thus
\begin{equation}\label{eq:appendaux}
 \sum_{n'} |c_{\mu,n'}|^2 \rho^C_{n',n'}  = |c_{\mu,n^*}|^2 - \beta_1^C~ \tfrac{b_{\mu,n^*}(\rho^C,\beta_1^C)}{\rho^C_{n^*,n^*}}
 .
\end{equation}
Taking $n\neq n^*$, by Assumption~\ref{assum:impt}, there exists $\mu$ such that $|c_{\mu,n}|^2\neq |c_{\mu,n^*}|^2$.
Replacing~\eqref{eq:appendaux} in~\eqref{eq:rhoCn} yields:
%\begin{multline*}
%\left(|c_{\mu,n}|^2-|c_{\mu,n^*}|^2 + \beta_1^C~ \tfrac{b_{\mu,n^*}(\rho^C,\beta_1^C)}{\rho^C_{n^*,n^*}}\right)  \rho^C_{n,n} \\= \beta_1^C~ b_{\mu,n}(\rho^C,\beta_1^C)
%\end{multline*}
$
\left(|c_{\mu,n}|^2-|c_{\mu,n^*}|^2 + \beta_1^C~ \tfrac{b_{\mu,n^*}(\rho^C,\beta_1^C)}{\rho^C_{n^*,n^*}}\right)  \rho^C_{n,n} = \beta_1^C~ b_{\mu,n}(\rho^C,\beta_1^C).
$
Thus, there exists $C_0>0$, such that for $n\neq n^*$ and $C$ large enough
$\rho^C_{n,n} \leq C_0 |\beta_1^C|.$
But $\rho^C_{n^*,n^*} = 1- \sum_{n\neq n^*} \rho^C_{n,n} \geq 1 - C_0(d-1) |\beta_1^C|$. This  is in contradiction with  $\rho^C_{n^*,n^*} \leq 1 - C |\beta_1^C|$ as soon as $C > C_0(d-1)$.
\end{document}